\documentclass[11pt,letterpaper]{amsart}

\usepackage{tikz}
\tikzstyle{every node}=[circle, draw, fill=black!50,
                        inner sep=0pt, minimum width=3pt]

\usepackage{amsfonts,amsmath,latexsym,color,epsfig,hyperref,enumitem, amssymb,bbm}

\newtheorem{theorem}{Theorem}[section]
\newtheorem{lemma}{Lemma}[section]
\newtheorem{claim}[lemma]{Claim}
\newtheorem{cor}[lemma]{Corollary}
\newtheorem{prop}[lemma]{Proposition}

\renewcommand{\le}{\leqslant}
\renewcommand{\ge}{\geqslant}

\def\qed{\ifvmode\mbox{ }\else\unskip\fi\hskip 1em plus 10fill$\Box$}
\makeatletter
\def\Ddots{\mathinner{\mkern1mu\raise\p@
\vbox{\kern7\p@\hbox{.}}\mkern2mu
\raise4\p@\hbox{.}\mkern2mu\raise7\p@\hbox{.}\mkern1mu}}
\makeatother

\def\R{\mathbb R}

\def\T{\mathbb T}
\def\mc{\mathcal }

\def\d{\delta}
\def\M{\operatorname{M}}

\newcommand{\blue}[1]{{\color{blue}#1}}

\usepackage{bbm}   
\usepackage{scalerel,stackengine}
\stackMath
\newcommand\widecheck[1]{%
\savestack{\tmpbox}{\stretchto{%
  \scaleto{%
    \scalerel*[\widthof{\ensuremath{#1}}]{\kern-.6pt\bigwedge\kern-.6pt}%
    {\rule[-\textheight/2]{1ex}{\textheight}}%WIDTH-LIMITED BIG WEDGE
  }{\textheight}% 
}{0.5ex}}%
\stackon[1pt]{#1}{\scalebox{-1}{\tmpbox}}%
}
\usepackage{ esint }

\title{Heilbronn's triangle problem in three dimensions}

\author{Dominique Maldague, Hong Wang, Dmitrii Zakharov}

\address{University of Cambridge, UK and University of California,
Los Angeles, USA}
\email{dm672@cam.ac.uk, dmal@math.ucla.edu}

\address{Courant institute of mathematical sciences, New York University}
\email{hw3639@nyu.edu}

\address{Department of Mathematics, Massachusetts Institute of Technology, Cambridge, USA}
\email{zakhdm@mit.edu}

\date{}

\begin{document}

\maketitle 

\begin{abstract}
    We show that among any $n$ points in the unit cube one can find a triangle of area at most $n^{-2/3-c}$ for some absolute constant $c >0$. This gives the first non-trivial upper bound for the three-dimensional version of Heilbronn's triangle problem. This estimate is a consequence of the following result about configurations of point-line pairs in $\R^3$: for $n \ge 2$ let $p_1, \ldots,p_n \in [0,1]^3$ be a collection of points and let $\ell_i$ be a line through $p_i$ for every $i$ such that $d(p_i, \ell_j) \ge \d$ for all $i\neq j$. Then we have $n \lesssim \d^{-3+\gamma}$ for some absolute constant $\gamma>0$.  The analogous result about point-line configurations in the plane was previously established by Cohen, Pohoata and the last author. 
\end{abstract}

\section{Introduction}\label{sec:intro}

In late 1940-s, Heilbronn asked the following question in discrete geometry. For a set of points $P$ let $\Delta=\Delta(P)$ be the minimum area of a triangle spanned by a triple of points in $P$ (collinear triples count as a 0 area triangles). Now define $\Delta(n) = \max_P \Delta(P)$  where the maximum is taken over all subsets $P \subset [0,1]^2$ of size $n$. How does the function $\Delta(n)$ grow with $n$? Two constructions (one algebraic and one probabilistic) due to Erd{\H o}s show that $\Delta(n) \ge c n^{-2}$ for some constant $c>0$ and Heilbronn conjectured that this lower bound is sharp. In 1982, Koml\'os--Pintz--Szemer\'edi \cite{komlos1982lower} disproved his conjecture by showing that $\Delta(n) \gtrsim \frac{\log n}{n^2}$. In the other direction, it is not hard to see that $\Delta(n) \le Cn^{-1}$ holds for some constant $C$. In 1951, Roth \cite{roth1951problem} managed to improve this bound by a small factor tending to zero using a clever density increment argument. Later on, in 1972-73, after a slight improvement by Schmidt \cite{schmidt1972problem}, Roth \cite{roth1972problem2}, \cite{roth1972problem3} developed a new powerful analytic method and used it to get a polynomial saving $\Delta(n) \lesssim n^{-1-c}$ for some (explicit) constant $c>0$. 
In 1981, Koml\'os--Pintz--Szemer\'edi \cite{komlos1981heilbronn} optimized Roth's method and proved that $\Delta(n) \lesssim n^{-8/7+o(1)}$. This was the best known upper bound for about 40 years, until in 2023-24, Cohen, Pohoata and the third author managed to improve the exponent from $8/7$ to $8/7+c$ in \cite{cohen2023new} and then to $7/6$ in \cite{cohen2025lower}. These improvements were made possible by discovering new connections between Roth's approach and recent work in harmonic analysis and fractal geometry. In particular, Roth's analytic method turns out to be closely related to the high-low method in harmonic analysis pioneered by Guth--Solomon--Wang \cite{guth2019incidence} and the key new input in the improved exponent $8/7+c$ in \cite{cohen2023new} comes from the work on radial projections estimates \cite{orponen2024kaufman}.

The following is a natural generalization of Heilbronn's triangle problem to higher dimensions. For integers $d\ge 2$ and $k\in [3, d+1]$ and a set of points $P \subset \R^d$ we can define $\Delta_k(P)$ to be the smallest $k-1$-dimensional volume of a $k$-vertex simplex defined by $P$. We then define a function $\Delta_{k,d}(n) = \max_{P} \Delta_k(P)$ where the maximum is over all subsets $P \subset [0,1]^d$ of size $n$. In particular, we have $\Delta_{3,2}(n)=\Delta(n)$. One can also consider values $k > d+1$ and let $\Delta_k(P)$ to be the  smallest possible $d$-volume of a convex hull of a $k$-tuple of points in $P$. For any given pair of parameters $k, d$ we can investigate the asymptotic behavior of the function $\Delta_{k, d}(n)$. 
Perhaps the two most interesting special cases of this question are the cases of the full-dimensional simplices, $k=d+1$, and triangles, $k=3$ (we refer to \cite{zakharov2024upper} for the treatment of other values of $k$). In the case of simplices, the `easy' bounds are $n^{-d} \lesssim \Delta_{d+1, d}(n) \lesssim n^{-1}$.
The lower bound was improved by Lefmann  \cite{lefmann2008distributions} to $\Delta_{d+1,d}(n) \gtrsim \frac{\log n}{n^d}$ and the upper bound was improved by the third author \cite{zakharov2024upper} to $\Delta_{d+1, d}(n) \lesssim n^{-\log d +C}$ for some constant $C$ (along with a number of bounds in low dimensions, e.g. that $\Delta_{4,3}(n) \le C n^{-4/3}$). 
In the case of triangles in $\R^d$, the `easy' bounds are $n^{- \frac{2}{d-1}} \lesssim \Delta_{3,d}(n) \lesssim n^{-\frac{2}{d}}$. The lower bound follows by taking a uniformly random set of $2n$ points and removing from it all triangles with area less than $c n^{-\frac{2}{d-1}}$ for a small constant $c>0$. This lower bound was improved to $\Delta_{3, d}(n)\gtrsim (\log n)^{\frac{1}{d-1}}n^{-\frac{2}{d-1}}$ by Lefmann \cite{lefmann2008distributions} following the semi-random approach of Koml\'os--Pintz--Szemer\'edi \cite{komlos1982lower}.
For the upper bound one can use the following packing argument. Since $[0,1]^d$ can be covered by at most $n/3$ cubes with sides of length at most $10n^{-1/d}$, by pigeonhole principle, there is a cube containing at least 3 points of $P$. These points then form a tringle of area less than $C n^{-2/d}$.
To the best of our knowledge, this is still asymptotically the best known upper bound on the function $\Delta_{3, d}(n)$ for any $d\ge 3$. 

We make the first progress in this direction and prove a new estimate for Heilbronn's triangle problem in $\R^3$:

\begin{theorem}\label{thm:triangles-in-R3}
    Let $P \subset [0,1]^3$ be a set of $n$ points. Then $P$ contains three points forming a triangle of area at most $C n^{-2/3 - c}$ for some absolute constants $c>0, C\ge 1$. 
\end{theorem}

In other words, we have the new upper bound $\Delta_{3,3}(n) \le C n^{-2/3-c}$.
Our proof produces an explicit, albeit very tiny, value of $c$. By tracking the estimates, one should be able to take $c= 10^{-100}$ or so in Theorem \ref{thm:triangles-in-R3}, however we did not make any attempt to optimize this value. 

In \cite{cohen2025lower}, the estimate $\Delta(n) \le n^{-7/6+o(1)}$ was proven using the following result about configurations of incident point-line pairs.
\begin{theorem}[\cite{cohen2025lower}]\label{thm:point-line-R2}
    For all $\varepsilon >0$ there is $C_\varepsilon$ such that the following holds for all $n\ge 2$. 
    Let $p_1, \ldots, p_n \in [0,1]^2$ be a collection of points and let $\ell_i$ be a line through $p_i$. Suppose that $d(p_i, \ell_j) \ge \d$ for all $i\neq j$ and some $\d>0$. Then we have $n \le C_\varepsilon \d^{-\varepsilon}\d^{-3/2}$. 
\end{theorem}
In the same spirit, we reduce Theorem \ref{thm:triangles-in-R3} to the following three dimensional version of Theorem \ref{thm:point-line-R2}:

\begin{theorem}\label{thm:main}
    There exists an absolute constant $\gamma > 0$ such that for all $\varepsilon >0$ there is $C_\varepsilon$ such that the following holds. 
    Let $p_1, \ldots, p_n \in [0,1]^3$ be a collection of points and let $\ell_i$ be a line through $p_i$. Suppose that $d(p_i, \ell_j) \ge \d$ for all $i\neq j$ and some $\d>0$. Then we have $n \le C_\varepsilon \d^{-\varepsilon} \d^{-3+\gamma}$.
\end{theorem}

% They then used this result to deduce the upper bound $\Delta(n) \le n^{-7/6+o(1)}$ on Heilbronn's triangle problem. 
% We can similarly use Theorem \ref{thm:main} to give a quick proof of Theorem \ref{thm:triangles-in-R3}:

\begin{proof}[Proof of Theorem \ref{thm:triangles-in-R3}]
    Let $P \subset [0,1]^3$ be a set of size $n$. By covering the unit cube with $10/n^{1/3}$-size cubes and applying pigeonhole principle, we can find a pair of points $\{p_1, q_1\} \subset P$ such that $d(p_1, q_1) \le Cn^{-1/3}$. Repeating this argument, we can then find a pair $\{p_2, q_2\} \subset P$ disjoint from $\{p_1, q_1\}$ so that $d(p_2, q_2) \le C(n-2)^{-1/3}$. Continuing in this manner,  we can find a collection of $m=[n/4]$ pairwise disjoint pairs of points $\{p_i, q_i\} \subset P$ such that $d(p_i, q_i) \le 2Cn^{-1/3}$ for all $i$. Let $\ell_i = \overline{p_iq_i}$ and consider the collection of point-line pairs $X = \{(p_i, \ell_i), ~i=1,\ldots, m\}$. Let $\d = \min_{i\neq j} d(p_i, \ell_j)$. 
    By Theorem \ref{thm:main}, for all $\varepsilon>0$ we have $m \lesssim_\varepsilon \d^{-3+\gamma-\varepsilon} $. Thus, there exists $i\neq j$ such that $d(p_i, \ell_j) \le \d\lesssim n^{ -\frac{1}{3-\gamma}+o(1) }$. The triangle $p_i p_jq_j$ has area at most $n^{-\frac{1}{3} -\frac{1}{3-\gamma}+o(1)} \le Cn^{-2/3-c}$, finishing the proof.
\end{proof}

To put results stated in Theorem \ref{thm:main} and Theorem \ref{thm:point-line-R2} into perspective, it is helpful to introduce some notation. A {\em point-line pair} in $\R^d$ is a pair $(p, \ell)$ where $p \in [0,1]^d$ and $\ell$ is a line passing through $p$. We denote by $\Omega_d$ the set of all point-line pairs in $\R^d$ (note that $\dim\Omega_d=2d-1$). We are interested in properties of finite collections of point-line pairs, i.e. subsets $X \subset \Omega_d$. For any such $X$, we define the {\em minimal distance} of $X$ as 
\[
d(X) = \min_{(p, \ell)\neq (p', \ell')\in X} d(p, \ell'),
\]
i.e. it is the minimal distance between a point in the configuration $X$ and a line passing through a different point in the configuration.

Let $\operatorname{PL}_d(\gamma)$ denote the following statement: for any $\varepsilon>0$ there is $C_\varepsilon$ such that for any point-line configuration $X \subset \Omega_d$ with $d(X) \ge \delta$, it follows that $|X| \le C_\varepsilon \d^{-\varepsilon}  \delta^{-d+\gamma}$ for every $\varepsilon>0$. 
With this definition, Theorem \ref{thm:point-line-R2} is the statement $\operatorname{PL}_2(1/2)$ and 
Theorem \ref{thm:main} is equivalent to the statement that $\operatorname{PL}_3(\gamma)$ holds for some $\gamma>0$.

Note that for any $d\ge 2$, the statement $\operatorname{PL}_d(0)$ is trivially true. Indeed, for a configuration $X \subset \Omega_d$, consider the set of points $P[X] = \{ p:~(p,\ell)\in X\} \subset [0,1]^d$ (taking points with multiplicity if there are overlaps). By covering the unit cube with subcubes of size $C_d|X|^{-1/d}$ and applying the pigeonhole principle, we can find some $p \neq p' \in P[X]$ so that $d(p, p') \le C_d |X|^{-1/d}$. Now if $(p', \ell') \in X$ is the corresponding point-line pair, then we get $d(X) \le d(p, \ell') \le d(p, p')\le C_d |X|^{-1/d}$. So if $d(X) \ge \d$ then it follows that $|X| \lesssim \d^{-d}$, showing $\operatorname{PL}_d(0)$. 

On the other hand, we can consider a collection of $\sim\d^{1-d}$ points in $[0,1]^{d-1}\times \{0\}$ which are pairwise $\d$-separated and for each point draw the line in the vertical direction through it. The resulting point-line configuration $X$ satisfies $|X| \sim \d^{1-d}$ and $d(X) \ge \d$. This implies that $\operatorname{PL}_d(1+\varepsilon)$ is false for every $\varepsilon>0$. In an upcoming note \cite{logunov2025configuration}, it will be shown that exists some constant $c>0$ such that for all $d\ge 2$ the statement $\operatorname{PL}_d(1-c)$ does not hold. 

In the plane, the statement $\operatorname{PL}_2(1/2)$ appears to be the limit of the high-low method as demonstrated by the Szemer\'edi--Trotter example and the hermitian unital example in the finite field model, we refer to \cite{cohen2025lower} for more discussion on this.
%In higher dimensions, heuristic arguments suggest that $\operatorname{PL}_d(1/d)$ might be an important threshold for analytic approaches. 
There are some tentative reasons to think that $\operatorname{PL}_d(1/d)$ is going to be at the limit for analytic approaches and reaching this limit would be very interesting. 
However the proof presented in this paper is quite far from this limit as the explicit value of $\gamma$ in Theorem \ref{thm:main} is extremely poor. The proof presented in this paper relies on a certain numerical coincidence specific to $\R^3$ (see the next section for more details) and it appears that the problem becomes significantly harder in dimension four. Thus, we leave proving that $\operatorname{PL}_4(\gamma)$ holds for some $\gamma>0$ as a direction for future research.

The rest of the paper is organized as follows. In Section \ref{sec:proof-outline} we provide an exposition of the proof of Theorem \ref{thm:main}. We start by giving a proof of $\operatorname{PL}_2(\gamma)$ for some $\gamma>0$ using the two-dimensional high-low inequality. We then discuss the adjustments to the strategy in three dimensions. In Section \ref{sec:incidence} we prove two results about incidences between points and tubes. These results are based on Wolff's two-ends and hairbrush arguments and do not require any `advanced' geometric measure theory (though those might become handy when trying to optimize the value of $\gamma$). 
In Section \ref{sec:high-low}, we prove several high-low inequalities for point-tube incidences in $\R^3$, starting from the basic estimate, followed by two refined versions. The incidence results from the previous section are a key ingredient here. In Section \ref{sec:preliminaries} we give some preliminary results needed in Section \ref{sec:main-proof}, in which we employ the high-low inequalities from Section \ref{sec:high-low} in an iterative scheme and prove Theorem \ref{thm:main}.

\subsection{Notation.}

We use the following asymptotic notation. We write $A\lesssim B$ if $|A|\le CB$ for some constant  $C$, we also write $A\sim B$ if $A\lesssim B$ and $B\lesssim A$ (and similarly for the following notation). 
We write $A\lesssim_K B$ if $A \le CK^CB$ for some constant $C$ and similarly if there more parameters in the subscript. 
We sometimes write $A \lessapprox_\varepsilon B$ to denote the fact that $A \lesssim_{\d^{-\varepsilon}} B$ where $\d>0$ is a parameter which will be clear from context. Sometimes we write $O(1)$ for a quantity bounded in absolute value by a constant and $o(1)$ for a quantity tending to $0$ with $\d$.
Finally, we somewhat informally use the notation $A\ll B$ to say that `$A$ is much smaller than $B$', e.g. it will typically mean that $A\lesssim \d^c B$ for some $c>0$ unless specified otherwise. 

\subsection{Acknowledgements.} The third author thanks Alex Cohen and Larry Guth for helpful conversations about the problem. 

\section{Proof overview.}\label{sec:proof-outline}

In this section we outline the main ideas in the proof of Theorem \ref{thm:main}. We begin by recalling the two-dimensional high-low inequality and use it to give a short proof of the two-dimensional version of Theorem \ref{thm:main}. Then we discuss the new difficulties that arise when we go to three dimensions and how to overcome them. 

\subsection{High-low inequality in two dimensions.}\label{subsec:high-low-R2} The key tool in polynomial upper bounds for Heilbronn's triangle problem is an analytic inequality relating the incidences between points and lines on different scales. We follow the setup used in \cite{cohen2023new} and \cite{cohen2025lower}. 

Let $P$ and $L$ be finite collections of points and lines in $[0,1]^2$. For a certain symmetric bump function $\varphi: \R\to [0,1]$ supported on $[-2,2]$ we define
\[
I(\d; P, L) = \sum_{p \in P, ~\ell \in L} \varphi( \d^{-1}d(p, \ell) ).
\]
This is the number of approximate incidences between $P$ and $L$ on scale $\d$, smoothed out by a bump function $\varphi$. It is essentially counting the number of pairs $(p, \ell) \in P\times L$ such that $d(p, \ell) \le \d$.

Note that if $P$ and $L$ are taken randomly inside $[0,1]^2$, then we expect $I(\d; P, L) \approx \d|P| |L|$. 
So it is natural to introduce the normalization 
\[
B(\d) = B(\d;P, L) = \frac{I(\d;P,L)}{\d|P| |L|}.
\]
Roth's key idea was to relate the incidence counts $I(\d;P, L)$ between many different scales. Roughly speaking, he observed that the normalized incidence function $B(\d)$ is essentially constant on scales where points and lines are sufficiently well distributed. Namely, let us define concentration numbers
\begin{align*}
&\M_P(\d) = \max_{Q:~\d\times \d \text{ square}} |P\cap Q|,\\
&\M_L(\d\times 1) = \max_{T:~\d\times 1\text{ tube}} |L\cap T|
\end{align*}
where $L\cap T$ is the set of lines $\ell \in L$ so that $|T\cap \ell| \ge 1/2$. With this notation, we have the following estimate.

\begin{prop}\label{prop:high-low-R2}
    There is a bump function $\varphi$ such that any $P$ and $L$ and $\d>0$ we have:
    \begin{equation}\label{eq:high-low-R2}
    |B(\d) - B(\d/2)|^2 \lesssim \d^{-3} \frac{\M_P(\d)}{|P|} \frac{\M_L(\d\times 1)}{|L|}
\end{equation}
\end{prop}

See Appendix A in \cite{cohen2025lower} for a proof. In Section \ref{sec:high-low} we will use the Fourier-analytic approach of Guth--Solomon--Wang \cite{guth2019incidence} to prove versions of (\ref{eq:high-low-R2}) in three dimensions. In particular, by following the proof of Lemma \ref{lem:easy-high-low} one can recover (\ref{eq:high-low-R2}) with an $\d^{-\varepsilon}$-loss. For now, let us just briefly describe the proof strategy of (\ref{eq:high-low-R2}). To each $P$ and $L$ we can associate (blurred) indicator functions $g$ and $f$. Then the incidence number $B(\d)$ can be computed as a scalar product $\langle f, g\rangle$ of functions $f$ and $g$. By applying Plancherel, we can rewrite this as a scalar product of their Fourier transforms $\hat f$ and $\hat g$. Now we can split the Fourier transform $\hat f$ into the high frequency and low frequency parts, giving us a decomposition $f = f^{high}+f^{low}$. The low frequency is obtained by restricting $\hat f$ on a ball of radius $\d^{-1}/K$ around the origin. On the physical side, this corresponds to blurring $f$ by a bump function of radius $K\d$. This means that the low part of the scalar product $\langle g, f^{low}\rangle$ is essentially the normalized incidence count at scale $K\d$, i.e. it is essentially $B(K\d)$. Now we can use this to estimate using Cauchy--Schwarz $|B(\d)- B(K\d)| \le \langle g, f^{high}\rangle \le \|g\|_2 \|f^{high}\|_2$. We then convert each $L_2$-norm using orthogonality and $L_1-L_\infty$ to information about distribution of $P$ and $L$ on scale $\d$. After some computation, we get that the high part is controlled by the right hand side of (\ref{eq:high-low-R2}) (with an extra $\d^{-\varepsilon}$ loss).
 
Let us close this section by giving three sharp constructions for (\ref{eq:high-low-R2}):
\begin{itemize}
    \item[(i)] Let $P=\{0\}$ and let $L$ be a collection of $\sim\d^{-1}$ directionally separated lines through the origin. Then we have $|B(\d) - B(2\d)| \sim B(\d) \sim \d^{-1}$ and $\d^{-3} \frac{\M_P(\d)}{|P|} \frac{\M_L(\d)}{|L|} \sim \d^{-2}$, matching (\ref{eq:high-low-R2}).

    \item[(ii)] Dually,  Let $L=\{\ell_0\}$ for some line $\ell_0$ and let $P$ be a collection of $\sim\d^{-1}$ points on $\ell_0$. Then we have $|B(\d) - B(2\d)| \sim B(\d) \sim \d^{-1}$ and $\d^{-3} \frac{\M_P(\d)}{|P|} \frac{\M_L(\d)}{|L|} \sim \d^{-2}$, matching (\ref{eq:high-low-R2}).

    \item[(iii)] A much less trivial example comes from sharp examples for the Szemer\'edi--Trotter theorem: let $P$ and $L$ be sets of $N$ points and lines defining $\sim N^{4/3}$ incidences. Then for $\d = c N^{-2/3}$ one can verify that both sides of (\ref{eq:high-low-R2}) are $\sim 1$. See \cite{cohen2023new} for more details. 
\end{itemize}

So to make an improvement over (\ref{eq:high-low-R2}) we would need to put a spacing constraint on $(P, L)$ which would separate it from all 3 types of examples. 
The Szemer\'edi--Trotter example appears to make this a very challenging task 
and this is the key bottleneck for further improvements on Heilbronn's triangle problem in the plane. 

Rather surprisingly, the situation changes in three dimensions: for a direct analogue of (\ref{eq:high-low-R2}) in $\R^3$ (which we will discuss below) we only have two types of sharp examples roughly matching (i) and (ii) above, but the grid-like examples giving (iii) are no longer sharp. This allows us to give simple criteria under which the simple high-low estimate can be improved. 

\subsection{Point-line problem in two dimensions.}\label{subsec:2d-case}
Let us now give a proof of the two dimensional version of Theorem \ref{thm:main}. 

%For a small parameter $\varepsilon>0$ we write $A \lessapprox_{\varepsilon} B$ if $A \le C \delta^{-C\varepsilon} B$ holds for some constant $C$. Let us also denote $A\approx_{\varepsilon} B$ if $A \lessapprox_\varepsilon B$ and $B\lessapprox_\varepsilon A$ hold at the same time. 

Let $X = \{(p_i \in \ell_i), i=1, \ldots, n\}$ be a collection of point-line incident pairs inside $[0,1]^2$. 
Let us define the {\em minimal distance} of $X$ as
\[
d(X) = \min_{(p, \ell) \neq (p', \ell') \in X} d(p, \ell').
\]

\begin{prop}\label{prop:main-in-R2}
    Let $\delta>0$ and $X$ be a collection of point-line pairs in $[0,1]^2$ with $d(X) \ge \delta$. Then we have $|X| \lesssim \delta^{-2+\gamma}$ for some absolute constant $\gamma>0$. In other word, the property $\operatorname{PL}_2(\gamma)$ holds for some $\gamma>0$.
\end{prop}

Recall that \cite{cohen2025lower} showed that $\operatorname{PL}_2(1/2)$ holds. Here we present a much simpler argument than the one given in \cite{cohen2025lower}. By tracking down the estimates one can check that the proof below gives something like $\gamma=0.1$. 

\begin{proof}
For the sake of contradiction let us assume that $|X| \approx_\gamma \delta^{-2}$ holds. Let $P = P[X]$ and $L=L[X]$ be the sets of points and lines of the point-line configuration $X$. Note that we have $|P| = |L| = |X|$.

The condition $d(X) \ge \delta$ implies that points in $P$ are $\delta$-separated: indeed, for $(p, \ell), (p', \ell') \in X$ we have $d(p, p') \ge d(p, \ell') \ge \d$.
So for any $w \in (\delta, 1)$ we get
\begin{equation}\label{eq:P-in-Q}
\M_P(w) = \max_{Q \text{ }w\text{-square}}|P \cap Q| \lesssim (w/\delta)^2 \lessapprox_\gamma w^2 |P|.    
\end{equation}
Similarly, lines in $L$ must be $\delta$-separated and for $w\in (\d, 1)$ we get 
\begin{equation}\label{eq:L-in-T}
\M_L(w\times 1) =\max_{T\text{ }w\times 1\text{ tube}}|L \cap T| \lesssim (w/\delta)^2 \lessapprox_\gamma w^2 |L|.    
\end{equation} 

So by Proposition \ref{prop:high-low-R2}, we obtain
\begin{equation}\label{eq:Bw-difference-R2}
    |B(w) - B(w/2)|^2 \lesssim w^{-3} \frac{\M_P(w)}{|P|}  \frac{\M_L(w\times 1)}{|L|} \lessapprox_\gamma w.
\end{equation}

By the assumption, we have  $d(p, \ell') \ge \delta$ for any $p\in P$ and $\ell' \in L$ for which $(p,\ell')\not\in X$. This means that $P$ and $L$ define an abnormally small amount of incidences on scale $\d$: for a small constant $c>0$, we have $I(c\delta; P, L) \lesssim |X| \ll \d |P| |L|$ and so $B(c\delta) \lesssim \frac{|X|}{\delta |X|^2} \lessapprox_\gamma \delta$. 

On the other hand, we are now going show that $B(C\delta^{1/2}) \approx_{\gamma} 1$ holds. Using (\ref{eq:Bw-difference-R2}) for $w \in (c\d, C\d^{1/2})$ this will then lead to a contradiction for small enough $\gamma$. 

Denote $\Delta = \delta^{1/2}$ and let $\T$ be the set of essentially distinct $\Delta\times 1$ tubes which cover $L$ and let $\mc Q$ be the set of disjoint $\Delta$-boxes $Q$ covering $P$. 
Using $|P| ,|L| \approx_\gamma \delta^{-2}$ we see that $|\T|, |\mc Q| \approx_\gamma \delta^{-1}$. After a passing to a large subset in $X$ we may assume that every $T\in \T$ contains $\approx_{\gamma} \d^{-1}$ lines $\ell \in L$ and every $Q\in \mc Q$ contains $\approx_\gamma\d^{-1}$ and points from $P$. Using this, we can lower bound 
\begin{equation}\label{eq:lower-bound-incidence-R2}
I(C\Delta; P,L) \gtrapprox_\gamma \d^{-2}  I(\mc Q, \T)   
\end{equation}
where $I(\mc Q, \T)= \#\{(Q, T):~ Q\cap T\neq\emptyset\}$.
Fix $Q \in \mc Q$ and let us estimate the number of $T\in \T$ which intersect with $Q$. Summing over all $Q$ will give us a lower bound on the number of incidences between $\mc Q$ and $\T$. For every point $p \in P\cap Q$ we have a unique line $\ell = \ell(p)$ so that $(p,\ell) \in X$. Every such line $\ell(p)$ is covered by some tube $T =T(p)\in \T$. In particular, we have $T(p) \cap Q \neq\emptyset$ for every $p \in P\cap Q$. For a fixed $T\in \T$ let $p_1,\ldots, p_m$ be the set of points in $P\cap Q$ such that $T(p_j) = T$ for every $j$. Note that the directions of lines $\ell(p_j)$, $j=1, \ldots, m$, lie in an arc of length $\sim \Delta$ determined by the main axis of $T$. So by pigeonhole principle we can find a direction $\theta_0 \in S^1$ so that $|\theta(\ell(p_j))-\theta_0| \le  \Delta/100$ for $\gtrsim m$ indices $j\in [m]$. Let $j, j'$ be any two such indices. Since $d(p_{j'}, \ell(p_j))\ge \delta$ and $p_j, p_{j'} \in Q$, simple geometry implies that $|\pi_{\theta_0}(p_{j'})-\pi_{\theta_0}(p_{j})| \ge \delta/2$, where $\pi_{\theta_0}$ is the orthogonal projection in the direction of $\theta_0$. So since $\pi_{\theta_0}(p_j) \in \pi_{\theta_0}(Q)$ for any such $j$, we conclude that $m \lesssim \Delta/\delta = \delta^{-1/2}$. Thus,
\[
\#\{T\in \T:~ T\cap Q\neq\emptyset \} \gtrsim |P \cap Q| / m \gtrapprox_\gamma \delta^{-1/2}.
\]
 It follows that $I(\mc Q, \T) \gtrapprox_\gamma \delta^{-3/2}$ and so by (\ref{eq:lower-bound-incidence-R2}),  $I(C\Delta; P, L)\approx_\varepsilon \Delta |P| |L|$, i.e. $B(C\Delta) \gtrapprox_\varepsilon 1$. By combining this with $B(c\d) \lesssim_\varepsilon \d$ and (\ref{eq:Bw-difference-R2}), we arrive at a contradiction for sufficiently small $\gamma$. 
\end{proof}

\subsection{Higher dimensions}
Let us see how the high-low approach would work in $\R^d$.
Given collections of points $P$ and lines $L$ in $[0,1]^d$, we can define incidence functions $I(\d; P, L) = \sum_{p\in P, \ell\in L} \varphi(d(p, \ell))$ and $B(\d) = \frac{I(\d;P, L)}{\d^{d-1} |P||L|}$ similarly to $\R^2$ (note that in $\R^d$ we expect $I(\d) \approx \d^{d-1} |P||L|$). We can also define the concentration numbers
\begin{align*}
&\M_P(\d) = \max_{Q:~\d\text{ cube}} |P\cap Q|,\\
&\M_L(\d\times\ldots \times\d\times 1) = \max_{T:~\d\times 1\text{ tube}} |L\cap T|
\end{align*}
With these definitions, a direct analogue of Proposition \ref{prop:high-low-R2}, is the following estimate.
\begin{prop}\label{prop:high-low-Rd}
    For an appropriate bump function $\varphi$ and all $P$ and $L$ in $\R^d$ we have
    \begin{equation}\label{eq:high-low-Rd}
        |B(\d) - B(2\d)|^2 \lesssim_\varepsilon \d^{-\varepsilon} \d^{-3(d-1)} \frac{\M_P(\d)}{|P|} \frac{\M_L(\d\times\ldots \times\d\times 1)}{|L|}.
    \end{equation}
\end{prop}

See Lemma \ref{lem:easy-high-low} for a proof of this bound in $\R^3$ (the proof in other dimensions is identical). An estimate like this was first proven by Guth--Solomon--Wang \cite{guth2019incidence} to study well-spaced collections of tubes in $\R^3$. The proof strategy goes along the same lines as the outline we gave in Section \ref{subsec:high-low-R2} for the two-dimensional estimate. The main difference is in the estimate of the $L_2$-norm of the high term $\|f^{high}\|_2$: in the plane one can use separation in Fourier space to to split $f^{high}$ into a sum of (essentially) pairwise orthogonal pieces $f^{high}_\theta$ (essentially by grouping lines according to their direction). Then each piece $f^{high}_\theta$ can be bounded by orthogonality using physical separation. In higher dimensions however, the corresponding pieces $f^{high}_{\theta}$ are no longer separated in the Fourier space. The Fourier transform of each $f^{high}_{\theta}$ looks like a great circle on the $(d-1)$-dimensional sphere $S^{d-1}$ and so the overlap pattern of these great circles controls the size of the $L_2$-norm. In general, we do not have that much information about the great circles, but by using a crude estimate on the overlap, one can use this geometric description to deduce (\ref{eq:high-low-Rd}).

There are two main cases where this estimate is sharp:
\begin{itemize}
    \item[(i)] Let $P = \{0\}$ and $L$ be a collection of $\sim \d^{-(d-1)}$ directionally separated lines passing though the origin. Then we have $|T_\d(\ell) \cap P| = 1 = \d^{-(d-1)}(\d^{d-1} |P|)$ and so
    $|B(\d) -B(2\d)| \sim B(\d) \sim \d^{-(d-1)}$. On the other hand, 
    \[ 
    \d^{-3(d-1)}\frac{\M_P(\d)}{|P|} \frac{\M_L(\d\times\ldots \times\d\times 1)}{|L|} \sim \d^{-2(d-1)} 
    \] 
    matching (\ref{eq:high-low-Rd}). 

    Furthermore, by taking a union of $c \d^{-(d-1)}$ random translates of this construction, we get a collection of $\sim \d^{-(d-1)}$ points and $\sim \d^{-2(d-1)}$ lines such that $|B(\d)-B(2\d)| \sim B(\d) \sim 1$ and the high-low error (\ref{eq:high-low-Rd}) is also $\approx_\varepsilon1$. 
    
    \item[(ii)] Let $P$ and $L$ be sets of $\sim\d^{-(d-1)}$ points and $\sim \d^{-2(d-2)}$ lines contained is a fixed hyperplane. Then we typically have $|T_\d(\ell) \cap P| \sim \d^{d-2}|P| = \d^{-1} (\d^{d-1}|P|)$ and so $|B(\d)-B(2\d)| \sim B(\d) \sim \d^{-1}$. On the other hand, 
    \[
    \d^{-3(d-1)}\frac{\M_P(\d)}{|P|} \frac{\M_L(\d\times\ldots \times\d\times 1)}{|L|} \sim \d^{-2},
    \] 
    matching (\ref{eq:high-low-Rd}).

    Furthermore, by taking a union of $c \d^{-1}$ random translates of this construction, we get a collection of $\sim \d^{-d}$ points and $\sim \d^{-2d+3}$ lines such that $|B(\d)-B(2\d)| \sim B(\d) \sim 1$ and the high-low error is also $\sim 1$. 
\end{itemize}

Now let $X= \{(p\in \ell)\}$ be a point-line configuration in $[0,1]^d$ and suppose that 
\[
d(X)=\min_{(p, \ell) \neq (p', \ell')\in X} d(p, \ell') \ge \d.
\] 
From this condition it follows that $P=P[X]$ is a $\d$-separated set of points and $L=L[X]$ is a $\d$-separated set of lines in $[0,1]^d$. So we get $|X| \le |P| \lesssim \d^{-d}$. Our goal would be to improve on this estimate by a polynomial factor $\d^\kappa$ for some constant $\kappa>0$. For the sake of contradiction let us suppose that $|X| \approx_\kappa \d^{-d}$. As in the two-dimensional case, we have very few incidences on scale $\d$:
\[
I(c\d; P, L) \le |X| \ll \d^{d-1} |P| |L|
\]
and so we get $B(c\d) = \frac{I(c\d; P, L)}{\d^{d-1}|P||L|} \lessapprox_\kappa \d$. 

On the other hand, let us assume that we managed to prove that $B(\d^{\alpha} ) \approx_\kappa 1$ for some $\alpha\in (0,1)$. We would like to use the high-low method to reach a contradiction by showing that the difference $|B(w)-B(w/2)|$ for all $w \in (\d, \d^\alpha)$ is much less than 1. For this approach to work, we at the very least have to be able to show that $|B(\d)-B(2\d)| \ll 1$. The high-low inequality (\ref{eq:high-low-Rd}) gives us:
\begin{equation}\label{eq:easy-high-low-fail}
|B(\d)-B(2\d)|^2 \lesssim_\varepsilon \d^{-\varepsilon} \d^{-3(d-1)} \frac{\M_P(\d)}{|P|} \frac{\M_L(\d\times\ldots \times\d\times 1)}{|L|} \approx_\kappa \d^{3-d},
\end{equation}
where we used the estimates $\M_P, \M_L \lesssim 1$ coming from the fact that points in $P$ and lines in $L$ are pairwise $\d$-separated.
If $d=2$ then this bound gives $|B(\d)-B(2\d)|\lesssim_\kappa \d^{1/2}\ll 1$ and this allowed us to run the high-low argument in Section \ref{subsec:2d-case}. But for $d\ge 3$ the right hand side is no longer small enough. Examples (i), (ii) above show that this bound is sharp in general so this is a serious obstruction to this approach. 
The situation in case of $d=3$ is quite special however: the right hand side of (\ref{eq:easy-high-low-fail}) exceeds $1$ only by a factor of the form $\d^{-C\kappa+o(1)}$ where $\kappa$ is our initial error term in the assumption $|X| \approx_\kappa \d^{-d}$. So if we could squeeze out a tiny saving from the bound (\ref{eq:high-low-Rd}) then it might give us a fighting chance. 

This idea turns out to work: in Section \ref{sec:high-low} we prove two different estimates which improve (\ref{eq:high-low-Rd}) by a polynomial factor for $d=3$ under certain geometric constraints on $L$. Note that in example (ii) there are lots of lines contained in a fixed plane. One can also construct similar examples where lines concentrate a lot in neighborhoods of several planes. So a necessary condition to improve on (\ref{eq:high-low-Rd}) is that $L$ does not have too many lines in a neighborhood of a plane. Let $\M_L(u\times w\times 1)$ be the maximum number of lines $\ell \in L$ contained in a $u\times w\times 1$ box $\Pi$ (in the sense that $|\Pi \cap \ell| \ge 1/2$). The example (ii) then satisfies $\M_L(\d\times 1\times 1) \sim \d^{-2}$ which is the maximum possible value for a set of $\d$-separated lines $L$. 

Luckily for us, a situation like this is impossible for the set of lines $L=L[X]$ coming from our point-line configuration $X$ satisfying $d(X) \ge \d$ and $|X| \approx_\kappa \d^{-3}$. 
Namely, using the 2-dimensional case of the point-line problem (i.e. that $\operatorname{PL}_2(3/2)$ holds by Theorem \ref{thm:point-line-R2}), we can get a good estimate on the number of lines near any plane. In particular, we can get a non-concentration estimate $\M_L(\d \times 1\times 1) \le \d^{-3/2+o(1)} \ll \d^{-2}$.
Indeed, for a $\d\times 1\times 1$ slab $\Pi$ we may essentially treat the set of pairs $(p, \ell) \in X$ such that $|\ell \cap \Pi|\ge 1/2$ as a point-line configuration in the plane. 

A slight generalization of this idea implies that we have the following spacing condition of the set of lines $L$:
\begin{equation}\label{eq:spacing-condition}
\M_L(u\times w\times 1) \ll (u/\d) \times (w/\d)^2, \text{ for all }\d\le u\le w \le 1
\end{equation}
(where $\ll$ refers to some polynomial gain, possibly involving powers of $u,w$, see Proposition \ref{prop:plane-reduction} for the precise claim).

Our first improvement over (\ref{eq:high-low-Rd}) for $d=3$ has the following shape:
\begin{equation}\label{eq:first-high-low-gain}
\text{If }|\theta(L)|_\d \ll \d^{-2} \text{ and }(\ref{eq:spacing-condition}),\text{holds, then }|B(\d)-B(2\d)|\ll1,    
\end{equation}
where $\theta(L) \subset S^2$ is the set of directions of lines $\ell\in L$, $|\cdot|_\d$ denotes the $\d$-covering number. 
Note that the restriction on the number of directions in (\ref{eq:first-high-low-gain}) implies that $(P, L)$ cannot look like the sharp bush example (i) (since the number of lines passing through a fixed point is upper bounded by the number of directions) 
and condition (\ref{eq:spacing-condition}) implies that $L$ cannot look like the sharp plane example (ii). 
One drawback of this estimate is that it only works if $L$ does not span all possible directions. There is no reason to expect that this is going to hold for the set of lines $L$ coming from a point-line configuration $X$ with $d(X)\ge \d$. The good news is that restrictions on the direction set appear naturally when we start rescaling into smaller boxed and this is related to how we use the information about directions in small boxes in the `initial estimate' argument from Section \ref{subsec:2d-case}.

The second improvement has the following shape:
\begin{align}\label{eq:second-high-low-gain}
    \text{If }L \text{ is well-spaced on scale }\d^{1/2}\text{ and }(\ref{eq:spacing-condition})\text{ holds, then }|B(\d)-B(2\d)|\ll1.
\end{align}
Here the condition `well-spaced on scale $\Delta=\d^{1/2}$' means that for any $\d^{1/2}\times \d^{1/2}\times 1$ tube $T$ we have an estimate $|T \cap L| \lessapprox \Delta^4 |L| $. In other words, we require an almost maximum possible number of $\Delta$-tubes to cover $L$. Note that the bush example (i) does not satisfy this property: we only need $\sim \Delta^{-2}$ $\Delta$-tubes to cover each bush and there are only $\sim \d^{-1}$ bushes overall. Again the condition (\ref{eq:spacing-condition}) makes sure that $L$ does not look like example (ii). The `well-spaced' condition originates from Guth--Solomon--Wang \cite{guth2019incidence} and is used there is a similar manner to improve the bound on the high term. 

Coming back to our point-line configuration $X$, there is no reason for lines to concentrate in any prisms or tubes, so
it is natural to expect that $L$ satisfies both premises of (\ref{eq:second-high-low-gain}). So at least in a `typical' situation we can now get the desired improvement in (\ref{eq:easy-high-low-fail}). However there is a significant complication: in order to get a contradiction and show $|X| \ll \d^{-3}$ we need to improve over the simple high-low estimate (\ref{eq:high-low-Rd}) for every $w \in (\d,\d^\alpha)$, not just the end-point case $w=\d$, and we also have not yet explained how do we obtain the initial estimate $B(\d^\alpha) \approx 1$. %Let us ignore these issues for the sake of argument.

If (\ref{eq:second-high-low-gain}) does not apply, then this means that $L$ is not well-spaced on scale $\Delta$, i.e. we can cover $L$ by $\lesssim_\kappa \Delta^{-4+\beta}$ $\Delta$-tubes for some $\beta>0$. Let $\T$ denote the family of tubes covering $L$ (by pigeonholing, we know that each tube contain about the same number of lines). It follows that for a typical $\Delta$-box $Q$ intersecting $P$ we then have $\approx_\kappa \Delta^2 |\T| \lesssim \Delta^{-2+\beta}$ tubes from $\T$ which intersect with $Q$ (cf. with the initial estimate part in Section \ref{subsec:2d-case}). Consider an affine rescaling map $\psi: Q\to [0,1]^3$ and let $P' = \psi(P\cap Q)$ and $X' = \{\psi(p\in \ell), ~p\in P\cap Q\}$ the corresponding point-line configuration. We then have $|X'| \approx_\kappa \Delta^{-3}$ and $d(X') \ge \Delta$, i.e. $X'$ is an essentially sharp example for the point-line problem on scale $\Delta$. On the other hand, observe that for $L' = L[X']$ we have
\[
|\theta(L')|_\Delta \lesssim \#\{T\in \T:~T\cap Q\neq\emptyset \} \lesssim \Delta^{-2+\beta}.
\]
So we constructed a new point-line configuration $X'$ essentially matching the trivial upper bound and which now determines a reduced amount of directions on scale $\Delta$. So we can use (\ref{eq:first-high-low-gain}) on $X'$ to improve the simple high-low estimate at scale $w=\Delta$. This is not the end of the story -- we still need to improve the high-low error for every scale $w>\Delta$. The restricted direction set property will ensure that (\ref{eq:first-high-low-gain}) is still applicable for $w$ within a small multiple of $\Delta$ (say $w \le \Delta^{1-\tau}$ for some small constant $\tau>0$). For larger $w$ we no longer have useful direction information, so we can only rely on the estimate (\ref{eq:second-high-low-gain}). If this estimate fails then this means that $L'$ is not well-spaced on scale $w^{1/2}$. So by rescaling into a $w^{1/2}$ box we get a new configuration $X''$ with restricted direction set on scale $w^{1/2}$. Continuing in this manner we slowly increase the range of scales where the bound (\ref{eq:first-high-low-gain}) applies and eventually we rescale into a final set $X^*$ for which we can apply (\ref{eq:first-high-low-gain}) for all $w \in [\d^*, \d^{*\alpha}]$ (where $\d^{*\alpha}$ is a scale on which we have available a sufficiently strong initial estimate). By carefully tracking the parameters, we show that we only need perform constantly many rescaling steps and so this procedure eventually terminates. In the end, we obtain a configuration $X^*$ for which we know that $B(\d^{*\alpha}) \approx_\kappa 1$, $B(\d^*) \lessapprox_\kappa \d^*$ and the high-low error is $\ll 1$ for every $w \in [\d^*, \d^{*\alpha}]$, which gives us a contradiction for sufficiently small $\kappa$. 
(In the actual proof in Section \ref{sec:main-proof}, this argument is arranged slightly differently in order to make tracking of parameters a little cleaner.)

%Proceeding in this way, we attempt to use either of (\ref{eq:first-high-low-gain}) or (\ref{eq:second-high-low-gain}) to improve the high-low error (\ref{eq:easy-high-low-fail}). In some cases, neither of the bounds applies and we are forced to pass to a rescaled point-line configuration in some box and start the high-low argument from scratch. 
%By carefully keeping track of the direction lines $L$ on different scales, we can ensure that we `make progress' after each iteration and that this procedure eventually terminates. This so after finitely many steps we can rescale to a configuration for which the high-low works at every scale and so we reach a contradiction with the assumption that $|X| \approx_\kappa \d^{-3}$.

%There is a complication we haven't addressed in this outline: for every high-low argument in this sketch we additionally need to establish the initial estimate $B(\d^{\alpha}) \approx_\kappa 1$. In the proof for the two-dimensional case above we essentially used the direction sets of $\d^{1/2}$-boxes $Q$ to prove a lower bound on $B(C\d^{1/2})$. Similar arguments work in higher dimensions as well. 

\section{Incidence estimates}\label{sec:incidence}

\subsection{Two-ends decomposition}

\noindent In the next section we prove high-low inequalities by analyzing neighborhoods of lines. On the Fourier side, $\d\times \d\times 1$ tubes become $\d^{-1}\times \d^{-1}\times 1$ slabs through the origin. If we study the overlap patterns of these slabs on $S^2$ then we see $\sim \d\times 1$ tubes on the surface of the sphere. The following lemma allows us to cut these tubes into a small number of pieces so that the resulting pieces do not overlap too much. This estimate on the overlap will be a key input into an $L^2$ estimate of the high frequency part in the proof of Theorem \ref{thm:refined_high-low}. 

This lemma is a version of Wolff's `two-ends' argument which gives an upper bound on the set of very rich points for a collection of tubes provided that those points are well-spread out on every tube. In our situation, we do not necessarily have such a condition in place and instead we reduce to such a condition but cutting out short segments where rich points concentrate too much.

\begin{lemma}\label{lem:two-ends}
    Let $1> \Delta > \delta > 0$. 
    For any collection of distinct $1\times \delta$ tubes $\T$, there exists a collection of essentially distinct $ 8\delta \times \Delta$ tubes $\mathbb U$ with the following property. For $T \in \T$ there is a set $\mathbb U(T) \subset \mathbb U$ of size at most $C \log_{2/\Delta} (1/\delta)$ such that the collection of sets
    \[
    2T \setminus \bigcup_{U \in \mathbb U(T)} U
    \]
    is at most $C \Delta^{-2} |\T|^{1/2}$ overlapping. Moreover, each $U \in \mathbb U(T)$ has the property that $T$ intersects the boundary of $U$ in the sides of length $8\d$. 
\end{lemma}

\begin{proof}
    Fix a set of tubes $\T$, without loss of generality we may assume that $\T$ is contained in $[-2,2]^2$. 
    For $i=1, \ldots, m$ and $T \in \T$ we will construct a sequence $U_{i}(T)$ of $\frac{1}{2}\Delta \times 4\delta$ tubes coaxial with $T$ as follows. For $i \ge 0$ suppose we already constructed $U_{j}(T)$ for $j \le i$. 

    Fix a $\delta/10$-net $\mc P$ in $[-2,2]^2$. For a collection of sets in the plane $\mathbb S$ let $\mc P_r(\mathbb S)$ denote the set of points $p \in \mc P$ which are contained in at least $r$ sets from $\mathbb S$. Fix $r = C_1 \Delta^{-2} |\T|^{1/2}$ for some constant $C_1$ which we will specify.

    For $T \in \T$ denote 
    \[
    T_i := 4T \setminus \bigcup_{j=1}^{i} U_{j}(T),
    \]
    and $T_0 = 2 T$. Given $T \in \T$, we define $U_{i+1}(T)$ to be a $\frac{1}{2}\Delta \times 4\delta$ tube $U$ coaxial with $T$ which maximizes the size of the intersection
    \[
    |\mc P_r(\T_i) \cap T_i \cap \frac{1}{2} U|.
    \]
    Here $\frac{1}{2} U$ refers to the $\frac{1}{4}\Delta \times 4\delta$ tube with the same center as $U$. 
    Now we inductively estimate the number of $r$-rich points for $\T_i$. For $i\ge 0$ we have by Cauchy--Schwarz
    \begin{equation}\label{eq:CS}
    I(\mc P_r(\T_i), \T_{i+1}) = \sum_{T \in \T} |\mc P_r(\T_i) \cap T_{i+1}| \le |\T|^{1/2} \left(\sum_{T\in \T} |\mc P_r(\T_i) \cap T_{i+1}|^2 \right)^{1/2}    
    \end{equation}
    
    For $T \in \T$ observe that by the choice of $U_{i+1}(T)$ we have 
    \[
    |\mc P_r(\T_i) \cap \frac{1}{2}U_{i+1}(T)| \ge \frac{1}{8}\Delta |\mc P_r(\T_i) \cap T_i| \ge \frac{1}{8}\Delta |\mc P_r(\T_i) \cap T_{i+1}|.
    \]
    Thus, we have
    \[
    |\mc P_r(\T_i) \cap T_{i+1}|^2 \le 8\Delta^{-1} |\mc P_r(\T_i) \cap T_{i+1}| |\mc P_r(\T_i) \cap \frac{1}{2}U_{i+1}(T)|
    \]
    and since $\frac{1}{2}U_{i+1}(T)$ and $T_{i+1}$ are $\Delta/4$-separated, we conclude:
    \[
    \sum_{T\in \T} |\mc P_r(\T_i) \cap T_{i+1}|^2 \le 8\Delta^{-1} \sum_{p, p' \in \mc P_r(\T_i):~ d(p,p') \ge \Delta/4} \#\left\{T \in \T:~ p \in T_{i+1}, ~ p' \in \frac{1}{2} U_{i+1}(T)\right\},
    \]
    there are at most $C\Delta^{-1}$ distinct $4\times 4\delta$ tubes passing through any pair of $\Delta/4$-separated points, so we get 
    \[
    \sum_{T\in \T} |\mc P_r(\T_i) \cap T_{i+1}|^2 \le C_0 \Delta^{-2}|\mc P_r(\T_i)|^2
    \]
    for some absolute constant $C_0$, thus,
    \[
    I(\mc P_r(\T_i), \T_{i+1}) \le C^{1/2}_0 \Delta^{-1} |\T|^{1/2} |\mc P_r(\T_i)|.
    \]
    So we get by the choice of $r = C_1 \Delta^{-2} |\T|^{1/2}$:
    \[
    |\mc P_r(\T_{i+1})| \le \frac{I(\mc P_r(\T_i), \T_{i+1})}{r} \le C_1^{-1} C_0^{1/2} \Delta |\mc P_r(\T_i)|.
    \]
    Let $C_1 = C_0^{1/2}$, then we obtain
    \[
    |\mc P_r(\T_i)| \le (\Delta/2)^{i} |\mc P_r(\T_0)| \le (\Delta/2)^{i} C\delta^{-2}
    \]
    so $\mc P_r(\T_m) = \emptyset$ for $m = C \log_{2/\Delta}(1/\delta)$.

    Now define $\mathbb U$ to be a collection of distinct $\Delta \times 8\delta$-tubes such that for every $T \in \T$ and $i =1, \ldots, m$ we have $U_i(T) \subset 0.9 U$ for some $U \in \mathbb U$. Define $\mathbb U(T) \subset \mathbb U$ to be a set of at most $m$ tubes covering $\{U_i(T)\}$. Observe that sets $2T \setminus \bigcup \mathbb U(T)$ have overlap at most $r$. Indeed, if there is a $r$-rich point for $\{T \setminus \bigcup \mathbb U(T)\}$ then the $\delta/10$-neighborhood of $p$ is $r$-rich for $\{4T \setminus \bigcup U_{i}(T)\}$. But this means that $\mc P_r(\T_m) \neq \emptyset$, contradicting the construction. This completes the proof.
\end{proof}

For reals $b \in [0,2\pi]$ and $a\in [0,b]$, we define a spherical $a\times b$-rectangle to be a region $R\subset S^2$ which is the $a$-neighbourhood of a great circle arc of length $b$. We will need an analogue of Lemma \ref{lem:two-ends} for spherical rectangles instead of tubes. For a spherical $a\times b$-rectangle $R$ we denote $c R$ the spherical $ca\times cb$ rectangle with the same axis and center. 

\begin{cor}\label{cor:two-ends}
    There is a constant $c>0$ such that the following holds for all $b \in (0,2\pi]$, $\Delta <c$ and $\d \le c\Delta b$. 
    For any collection of essentially distinct spherical $\delta\times b$-rectangles $\T$, there exists a collection of essentially distinct spherical $10\delta \times \Delta b$-rectangles $\mathbb U$ with the following property. For $T \in \T$ there is a set $\mathbb U(T) \subset \mathbb U$ of size at most $C \log_{2/\Delta} (1/\delta)$ such that the collection of sets
    \[
    1.5T \setminus \bigcup_{U \in \mathbb U(T)} U
    \]
    is at most $C \Delta^{-2} |\T|^{1/2}$ overlapping. Moreover, each $U \in \mathbb U(T)$ has the property that $U\cap T$ contains a spherical $\d \times (\Delta b-10\d)$-rectangle. 
\end{cor}

Here by essentially distinct rectangles we mean that $|U\cap U'| \le 0.9 |U|$ for every $U\neq U'$.

\begin{proof}
    Let $r>0$ be a small constant. Note that a spherical $\d\times b$ rectangle can be covered by at most $16/r$ many spherical $\d\times \min(r/8, b)$-rectangles. So by replacing rectangles in $\T$ by the covering we may assume that $b \le r/8$ holds (at the cost of increaing $C$ by a factor of $O(1/r)$). 
    
    We can cover $S^2$ by $\sim r^{-2}$ many $r$-balls $B_i$ in such a way that any $r/2$-ball is covered by at least one $B_i$. So is covered $\T$ by families $\T_{B_i} = \{T\in \T:~ T\subset B_i\}$ and we have the property that if $T, T'\in \T$ satisfy $T\cap T'\neq \emptyset$ then there is some $i$ so that $T, T' \in \T_{B_i}$. Thus, it is enough to prove the statement for one $\T_{B_i}$ at a time. Now let $\Pi$ be the plane tangent to $S^2$ at the center point of $B_i$ and consider the radial projection map $\pi:B_i \to \Pi$ (with projection center at the origin). Observe that under this map, spherical $a\times b$-rectangles contained in $B_i$ essentially map to $a\times b$ rectangles in the plane $\Pi$. More precisely, for any $\varepsilon>0$ and $r< r(\varepsilon)$ the image $\pi(R)$ of an $a\times b$ spherical rectangle is contained in a $(1+\varepsilon)a\times (1+\varepsilon)b$ plane rectangle and contains a $(1-\varepsilon)a\times (1-\varepsilon)b$ plane rectangle. So using the projection $\pi$ we may apply Lemma \ref{lem:two-ends} to a family of $(1-\varepsilon)\d \times (1-\varepsilon)b$ rectangles obtained from $\pi(\T_{B_i})$. Taking $\varepsilon$ a small constant, say $\varepsilon=0.1$, we get the desired conclusion. 
\end{proof}

% \remark{ The same statement holds if we replace tubes by rectangles on the sphere $S^2$, i.e. neighborhoods of great circle segments. This follows by cutting $S^2$ into constantly many pieces and projecting each piece radially onto a tangent plane to $S^2$. After this operation, spherical rectangles essentially become tubes. \blue{make this a corollary}}

\subsection{Hairbrush estimate.} 
For a set of $\delta\times 1$ tubes in the plane $\T$ be say that $\T$ is $(t, C)$-Katz-Tao of tubes, if for every $w\le 1$ and every $w\times 1$ box $B$ we have
\[
|\T \cap B| \le C (w/\delta)^{t}.
\]
This is a convenient spacing condition for studying incidence geometry of tubes and it appears naturally in the application to our point-line incidence problem. 

We say that a set of $\delta\times 1$-tubes $\T$ in $\R^3$ satisfies the Katz-Tao Convex Wolff axiom with exponents $(t_1, t_2)$ and error $C$ if for any $u\le w\le 1$ and $u\times w\times 1$ box $B$ we have 
\[
|\T \cap B| \le C (u/\delta)^{t_1} (w/\d)^{t_2}.
\]
For shorthand, we say $\T$ is a $(t_1, t_2, C)$-Katz-Tao set of tubes. Katz--Tao and similar notions of concentration were introduced in the works of Katz and Tao \cite{katz2001some} and are a ubiquitous tool in fractal geometry and harmonic analysis. The notion of $(t_1, t_2, C)$-Katz--Tao sets of tubes is closely related to Wolff axioms \cite{wolff1995improved} which is often used to study Kakeya and Furstenberg sets \cite{wang2024restriction}. In particular, Convex Wolff Katz--Tao axioms used in the recent resolution of the Kakeya conjecture in $\R^3$ \cite{wang2025volume} correspond to $(1, 1, C)$-Katz--Tao sets.

In Proposition \ref{prop:plane-reduction} below we show that the set of $\delta$-tubes defined by a point-line configuration $X$ with $d(X) \ge \delta$ forms a $(1+\gamma, 2-\gamma, \delta^{-\varepsilon})$-Katz-Tao set of tubes where $\gamma \in [0,1]$ is a number such that $\operatorname{PL}_2(\gamma)$ holds. In this section we use Wolff's hairbrush argument originating from \cite{wolff1995improved} to prove an incidence estimate on tubes satisfying the $(t_1, t_2, C)$-Katz--Tao condition. Our argument only relies on the `classical' techniques and does not use any of the recent developments in the area. This makes the paper essentially self-contained but leaves open a direction for potential optimization of our methods.

% Given a collection of tubes $\T$ with shading $Y(T) \subset T$ for $T\in \T$, we define the multiplicity of the shading as
% \[
% \mu(\T) = \mu(\T, Y) = \frac{\sum_{\T} |Y(T)|}{|\bigcup_\T Y(T)|},
% \]
% i.e. this is the average number of shadings $Y(T)$ containing a fixed point of the union $U=\bigcup_\T Y(T)$. By refining the shading, we can always ensure that in fact, every point of $U$ is contained in $\sim \mu$ shadings $Y(T)$.
Given a set of tubes $\T$ a shading $Y$ on $\T$ is a collection of subsets $Y(T) \subset T$ for every $T\in \T$. We say that $Y(T)$ is $\lambda$-dense if $|Y(T)| \ge \lambda |T|$. 
First we prove an estimate on the volume of the union of a shading $\bigcup_{T\in \T} Y(T)$ in terms of shading density provided that $\T$ is $t$-Katz--Tao. The estimate we prove follows from Wolff's two-ends argument and is far from optimal. The Furstenberg set estimate \cite{ren2023furstenberg} can be used to prove a sharp estimate. Since it is not necessary for our application we opted for a more elementary argument instead. 

\begin{lemma}\label{lem:planar-brush}
    Let $\T$ be a set of $\delta$-tubes in the plane which is a $(t, K)$-Katz-Tao for some $t\ge 1$ and $K>0$. Let $Y(T)\subset T$ be a $\lambda$-dense set for some $\lambda \in [\delta, 1]$. Then for every $\varepsilon >0$:
    \[
    \left|\bigcup_{T\in \T} Y(T)\right| \gtrsim_\varepsilon \delta^{\varepsilon} K^{-1}\lambda^2 \d^t |\T|
    \]
    % \[
    % \mu(\T, Y) \lesssim K\lambda^{-1} \delta^{1-t} \log^{O(1)} (1/\delta).
    % \]
\end{lemma}

\begin{proof}
    By dyadic pigeonholing, we can find a dyadic number $\mu$ and a subset $Y'(T) \subset Y(T)$ for every $T\in\T$ so that each point in $U'=\bigcup_\T Y'(T)$ is contained in $\sim \mu$ shadings $Y'(T)$ and so that we have $\sum |Y'(T)| \gtrsim \frac{1}{\log (1/\delta)} \sum |Y(T)|$. By double counting, we have 
    \begin{equation}\label{eq:double-count-mu}
    \mu |U'| \sim \sum_{T\in \T} |Y'(T)|\gtrsim \frac{1}{\log (1/\delta)} |\T| \lambda \delta.    
    \end{equation}
    Next, we perform a two-ends reduction. Let $\varepsilon>0$ be arbitrarily small. We may assume that $\mu \ge C \delta^{-\varepsilon}$ since otherwise (\ref{eq:double-count-mu}) already gives a sufficiently good bound. For every $p \in U'$ let $\T_{Y'}(p)$ be the set of tubes $T\in \T$ so that $p\in Y'(T)$. For $p\in U'$, let $T_{p, \theta}$ be a $1\times \theta$ tube for some dyadic $\theta\in [\delta, 1]$ which maximizes the quantity
    \begin{equation}\label{eq:define-T-p}
    |\T_{Y'}(p) [T_{p}]| \theta^{-\varepsilon},    
    \end{equation}
    where $\T_{Y'}(p) [T_{p}]$ consists of tubes $T\in \T_{Y'}(p)$ so that $|T\cap T_{p}| \ge |T|/2$. Note that we have by definition
    \begin{equation}
    |\T_{Y'}(p) [T_{p}]| \gtrsim \theta^{\varepsilon} |\T_{Y'}(p)| \ge \delta^{\varepsilon} |\T_{Y'}(p)|.    
    \end{equation}
    Let $Y''(T) \subset Y'(T)$ be a new shading where we include $p$ in $Y''(T)$ iff $T \in\T_{Y'}(p) [T_{p}]$. Then we get
    \[
    \sum |Y''(T)| = \int_{U'} |\T_{Y'}(p)[T_p]| \gtrsim \delta^{\varepsilon} \int_{U'} |\T_{Y'}(p)| = \delta^{\varepsilon}  \sum |Y'(T)|. 
    \]
    Finally, we can pigeonhole the value of $\theta$ and define a shading $\tilde Y(T) = Y'' (T) \cap U'_\theta$ where $p\in U'_\theta$ if $T_p$ has dimensions $1\times \theta$. By choosing $\theta$ appropriately, we can guarantee that $\sum|\tilde Y(T)| \gtrsim \frac{1}{\log(1/\delta)}\sum |Y''(T)|$. After these reductions, it suffices to lower bound the volume of $\tilde U=\bigcup \tilde Y(T)$. 

    Note that for any $p\in \tilde U$, the set of tubes $\T_{\tilde Y}(p)$ is contained in some $C\theta\times C$ tube (for some constant $C$). 
    Using this, we can construct a collection of $C\theta\times C$ tubes $\T_\theta$ with the following two properties: for every $T, T' \in \T$ so that $\tilde Y(T)\cap \tilde Y(T') \neq\emptyset$ there exists $T_\theta\in \T_\theta$ so that $T, T' \subset T_\theta$, and for every $T\in \T$ there are only constantly many $T_\theta\in \T_\theta$ so that $T\subset T_\theta$. For $T_\theta \in \T_\theta$ let $\T[T_\theta]$ be the set of tubes contained in $T_\theta$ and $\tilde U_{T_\theta} = \bigcup_{\T[T_\theta]} \tilde Y(T)$. We get $|\tilde U| \gtrsim \sum_{\T_\theta} |\tilde U_{T_\theta}|$. 

    Consider the following set
    \[
    S=\{(p, T, T')\in \tilde U\times \T\times \T:~ \exists T_\theta\in \T_\theta:~ T, T' \in \T[T_\theta],~p\in \tilde Y(T)\cap \tilde Y(T'), ~\angle T, T' \ge \alpha\theta\}
    \]
    where $\alpha = 2^{-C\varepsilon^{-1}}$ is a constant depending on $\varepsilon$. Using the maximality of $T_p$ in (\ref{eq:define-T-p}), every $\alpha\theta\times 1$ tube $T'_p$ contains at most $\alpha^{\varepsilon} |\T_{Y'}(p)[T_p]|$ tubes $T \in \T_{Y'}(p)$. By definition, for $p \in \tilde U$ we have $\T_{Y'}(p)[T_p] = \T_{Y''}(p) = \T_{\tilde Y}(p)$,
    so for every $p \in \tilde U$, we get that there are at least 
    $|\T_{\tilde Y}(p)| (|\T_{\tilde Y}(p)|- C\alpha^\varepsilon |\T_{\tilde Y}(p)|) \ge |\T_{\tilde Y}(p)|^2/2$ pairs $(T, T') \in \T_{\tilde Y}(p)\times \T_{\tilde Y}(p)$ which form an angle at least $\alpha \theta$. We conclude that 
    \begin{align*}
    |S| := \int_{\tilde U} \#\{(T, T'):~ (p,  T, T')\in S\} \ge \frac12\int_{\tilde U} |\T_{\tilde Y}(p)|^2\\
    \ge \frac12 \frac{1}{|\tilde U|}\left(\int_{\tilde U} |\T_{\tilde Y}(p)|\right)^2 = \frac12 \frac{1}{|\tilde U|} \left(\sum_{T\in \T} |\tilde Y(T)|\right)^2\\
    \gtrsim \d^{3\varepsilon} \frac{1}{|\tilde U|} \lambda^2 \d^2|\T|^2.
    \end{align*}
    On the other hand, for any $T,T'$ forming an angle at least $\alpha\theta$, we have $|\tilde Y(T)\cap \tilde Y(T')| \le |T\cap T'| \lesssim \delta^2/\alpha\theta$. So we obtain
    \[
    |S| \lesssim \sum_{T_\theta\in \T_\theta} |\T[T_\theta]|^2 (\delta^2/\alpha\theta) \lesssim_\varepsilon (\delta^2/\theta) |\T| \max_{T_\theta\in \T_\theta} |\T[T_\theta]|.
    \]
    By the Katz--Tao axiom, we have  $|\T[T_\theta]| \lesssim  K (\theta/\delta)^{t}$. So by combining upper and lower bounds on $S$, we obtain
    \[
     \d^{3\varepsilon} \frac{1}{|\tilde U|} \lambda^2 \d^2|\T|^2 \lesssim (\delta^2/\theta) |\T| K (\theta/\delta)^{t}
    \]
    \[
    |\tilde U| \gtrsim \delta^{3\varepsilon}K^{-1} \lambda^2  \theta^{1-t}\d^t |\T| \gtrsim  \delta^{3\varepsilon}K^{-1} \lambda^2 \d^t |\T| 
    \]
    since $t\ge 1$, as desired.
\end{proof}

It will be convenient to use the following corollary.
\begin{cor}\label{cor:plane-brush-cor}
    Let $u \in (\d, 1)$ and let $\T$ be a set of $\delta$-tubes contained in a $Cu\times C$ rectangle $R$ and which is a $(t, K)$-Katz-Tao for some $t\ge 1$ and $K>0$. Let $Y(T)\subset T$ be a $\lambda$-dense set for some $\lambda \in [\delta, 1]$. Then for every $\varepsilon >0$:
    \[
    \left|\bigcup Y(T)\right| \gtrsim_\varepsilon \delta^{\varepsilon} K^{-1}\lambda^2 \d^tu^{1-t} |\T|
    \]
\end{cor}
\begin{proof}
    Let $\psi: R\to [0, C]^2$ be the natural rescaling map and consider the set of $\sim\delta/u$-tubes $\T'=\psi(\T)$. It is easy to check that $\T'$ is $(C'K, t)$-Katz-Tao, so Lemma \ref{lem:planar-brush} gives 
    \[
    |\bigcup \psi(Y(T))| \gtrsim_\varepsilon \d^\varepsilon K^{-1} \lambda^2 (\d/u)^t |\T|
    \]
    Since $\psi$ increases volume by a factor of $u^{-1}$, the result follows.
\end{proof}

Now we use Lemma \ref{lem:planar-brush} (and actually Corollary \ref{cor:plane-brush-cor}) to prove a volume estimate for shadings on a $(t_1, t_2, C)$-Katz--Tao set of tubes in $\R^3$. This estimate is a generalization of Wolff's hairbrush \cite{wolff1995improved} estimate which was developed by Wolff to prove that Kakeya sets in $\R^3$ have dimension at least $5/2$. Needless to say that this estimate is far from being optimal and sharp volume estimates for $(t_1, t_2, C)$-Katz--Tao sets are currently unavailable: even in case of Kakeya $t_1=t_2=1$ it is essentially equivalent to the maximal Kakeya conjecture. 

\begin{lemma}\label{lem:space-brush}
    Let $\T$ be a set of $\delta$-tubes in $\R^3$ which is $(t_1, t_2, K)$-Katz-Tao for some $t_1, t_2\ge 1$ and $K>0$. Let $Y$ be a $\lambda$-dense shading on $\T$, then we have 
    \begin{equation}\label{eq:hairbrush}
    \left|\bigcup Y(T)\right| \gtrsim_\varepsilon \d^{\varepsilon} K^{-\frac{2+t_1}{2t_1+2t_2}} \lambda^{5/2} \d^2 |\T|^{\frac{2+t_1}{2t_1+2t_2}}. 
    \end{equation}
\end{lemma}

\begin{proof}
By repeating the dyadic pigeonholing and two-ends reduction from Lemma \ref{lem:planar-brush}, we can find a dyadic scale $\theta \in [\delta, 1]$ and a subset $\tilde Y(T) \subset Y(T)$ for every $T\in \T$ so that for every $p \in \tilde U = \bigcup_{T\in\T} \tilde Y(T)$ there is a $\theta$-tube $T_p$ so that $\T_{\tilde Y}(p) \subset \T_{Y}(p)[T_p]$ and $T_p$ maximizes (\ref{eq:define-T-p}). As previously, define a collection of $C\theta\times C$-tubes $\T_\theta$ covering $\T$ and define $\tilde U_{T_\theta} = \bigcup_{\T[T_\theta]} \tilde Y(T)$. After these refinements, we will have $|\tilde U| \gtrsim\sum_{\T_\theta} |\tilde U_{T_\theta}|$ and $\sum |\tilde Y(T)| \gtrsim \delta^{\varepsilon} (\log 1/\d)^{-O(1)} \sum |Y(T)|$.

For a $C\d\times C\theta\times C$ slab $H$ we can view the set of tubes $\T[H]$ contained in $H$ as essentially a planar configuration of tubes: we can slice $H$ by a random plane $P$ parallel to the two major axes of $H$ and study the intersections $\tilde Y(T)\cap P$, $T\in \T[H]$. 
Note that for any $C\delta \times Cw\times C$ subslab $B \subset H$ the $(t_1, t_2, K)$-Katz--Tao axiom implies $|\T[B]| \lesssim K (w/\d)^{t_2}$. So we get that the collection of tubes we obtain inside the plane $P$ is $(t_2, CK)$-Katz--Tao. So for any $H$ and subsets $Y'(T) \subset T$ we get the following estimate using Corollary \ref{cor:plane-brush-cor}:
\begin{equation}\label{eq:volume-TH}
|\bigcup_{T\in \T[H]} Y'(T)| \gtrsim \int_0^{C\delta} |\bigcup_{T\in \T[H]}  Y'(T) \cap P_x| dx \gtrsim_{\varepsilon} \delta^{\varepsilon} K^{-1} \delta^{1+t_2} \theta^{1-t_2} |\T[H]| \lambda_{H, Y'}^2    
\end{equation}
where $P_x$ denotes a plane distance $x$ away from a fixed face of $H$ and $\lambda_{H, Y'}$ is the average density of shadings $Y'(T)$ over $T\in \T[H]$. 

Now we fix $T_\theta\in \T_\theta$. Let $r = \delta^{C\varepsilon} \lambda$ be a parameter.
Similarly to the proof of Lemma \ref{lem:planar-brush} consider the following set:
\[
    S_{T_\theta}= \{(p, T, T')\in \tilde U_{T_\theta}\times \T[T_\theta]\times \T[T_\theta]:~ p\in \tilde Y(T)\cap \tilde Y(T'), ~\angle T, T' \ge \alpha\theta\}.
\]
Then we analogously to before, get for $\alpha = 2^{-C\varepsilon^{-1}}$:
\begin{equation}\label{eq:Stheta-lower-bound}
|S_{T_\theta}| \ge \frac{1}{2} \int_{\tilde U_{T_\theta}} |\T_{\tilde Y}(p)|^2 \gtrsim \frac{1}{|\tilde U_{T_\theta}|} \left(\sum_{T\in \T[T_\theta]} |\tilde Y(T)|\right)^2 \sim \frac{1}{|\tilde U_{T_\theta}|} \lambda_{T_\theta}^2 \d^4|\T[T_\theta]|^2,    
\end{equation}
where $\lambda_{T_\theta}$ is the average density of shadings $\tilde Y(T)$ for $T\in \T[T_\theta]$. 

For every $T\in \T[T_\theta]$ let $\T(T)$ be the set of $T'\in \T[T_\theta]$ such that $\tilde Y(T)\cap \tilde Y(T')\neq \emptyset$, $\angle T,T' \ge \alpha \theta$. 
%Further, let $\T'(T)$ be the set of $T'$ as above and such that we also have $|\tilde Y(T')| > C r |T'|$. Note that the latter condition guarantees that $|\tilde Y(T') \setminus \mc N_{r\theta}(T)| \ge \frac12 |\tilde Y(T')|$ holds. 
By definition, we have
\begin{equation}\label{eq:Stheta-upper-bound}
|S_{T_\theta}| = \sum_{T\in \T[T_\theta]} \sum_{T'\in \T(T)} |\tilde Y(T)\cap \tilde Y(T')| \lesssim \sum_{T\in \T[T_\theta]} |\T(T)| (\delta^3 / \alpha\theta).    
\end{equation}
Note that for $T'\in \T(T)$ we have $|\tilde Y(T') \setminus \mc N_{r\theta}(T)| \ge |\tilde Y(T')| - C\delta^2  r/\alpha$, where $\mc N_{r\theta}(T_0)$ is the $r$-neighbourhood of $T$.

For a tube $T \in \T[T_\theta]$ let us consider a collection of $C\d\times C\theta\times C$ slabs $H_1, \ldots, H_m$, $m \sim \theta/\d$, which pass through $T$ and cover $T_\theta$. Consider the shading $Y'(T') = \tilde Y(T') \setminus \mc N_{r\theta}(T)$. Applying (\ref{eq:volume-TH}) to each $\T(T)[H_i]$ with shading $Y'$ gives:
\begin{align}\label{eq:tildeUTtheta-lower}
|\tilde U_{T_\theta}| \gtrsim r\sum_{i=1}^m |\bigcup_{\T(T)[H_i]} Y'(T')| \gtrsim \delta^{\varepsilon}r K^{-1}\d^{1+t_2}  \theta^{1-t_2} \sum_{i=1}^m |\T(T)[H_i]|\lambda^2_{H_i, Y'} \\  \gtrsim \delta^{\varepsilon}r K^{-1}\d^{1+t_2}  \theta^{1-t_2} |\T(T)|\lambda^2_{T,Y'} \nonumber
\end{align}
where we used the fact that sets $H_i \setminus \mc N_{r\theta}(T_0)$ are $\lesssim r^{-1}$ overlapping. Here $\lambda_{T, Y'}$ denotes the average density of the shading $Y'(T')$ for $T' \in \T(T)$ and we used Cauchy--Schwarz to convert the sum of squares of densities $\lambda_{H_i, Y'}^2$ to $\lambda_{T, Y'}^2$. To be more precise, we estimate it as follows: note that by definition
\[
|\T(T)| \lambda_{T, Y'} = \sum_{i=1}^m |\T(T)[H_i]| \lambda_{H_i, Y}
\]
and so 
\begin{align*}
|\T(T)|^2 \lambda^2_{T, Y'} = \left(\sum_{i=1}^m |\T(T)[H_i]| \lambda_{H_i, Y}\right)^2 = \left(\sum_{i=1}^m |\T(T)[H_i]|^{1/2} (|\T(T)[H_i]|^{1/2} \lambda_{H_i, Y})  \right)^2 \\
\le |\T(T)| \sum_{i=1}^m |\T(T)[H_i]|\lambda_{H_i, Y}.
\end{align*}
We deal with the density terms in the following computations in a similar manner. 
So by (\ref{eq:Stheta-upper-bound}), we get
\[
|\T[T_\theta]| |\tilde U_{T_\theta}| \gtrsim \delta^{\varepsilon}r K^{-1}\d^{1+t_2}  \theta^{1-t_2}\sum_{T\in \T[T_\theta]} |\T(T)|\lambda^2_{T,Y'} \gtrsim \delta^{\varepsilon}r K^{-1}\d^{1+t_2} \theta^{1-t_2} |S_{T_\theta}| (\alpha\theta/\d^3) \lambda_{T_\theta, Y'}^2
\]
and so (\ref{eq:Stheta-lower-bound}) gives
\begin{align*}
     |\T[T_\theta]| |\tilde U_{T_\theta}| \gtrsim \delta^{\varepsilon}r K^{-1}\d^{1+t_2}  \theta^{1-t_2} (\alpha\theta/\d^3) \lambda_{T_\theta}^2 \lambda_{T_\theta, Y'}^2 \frac{1}{|\tilde U_{T_\theta}|} \d^4 |\T[T_\theta]|^2
\end{align*}
\begin{equation}\label{eq:Utheta-main-estimate}
|\tilde U_{T_\theta}|^2 \gtrsim \delta^{\varepsilon}r K^{-1}\d^{2+t_2}\theta^{2-t_2}|\T[T_\theta]|\lambda_{T_\theta}^2 \lambda_{T_\theta, Y'}^2.    
\end{equation}

% We have another easier estimate $|\tilde U_{T_\theta}|$ which follows from (\ref{eq:Stheta-lower-bound}) and (\ref{eq:Stheta-upper-bound}) and the trivial $|\T(T)| \lesssim |\T[T_\theta]|$, leading to:
% \begin{equation}\label{eq:Utheta-weak-estimate}
% |\tilde U_{T_\theta}| \gtrsim \lambda_{T_\theta}^2 \d \theta    
% \end{equation}

% Choose some coefficients $\beta_1, \beta_2\ge0$ such that $2\beta_1+\beta_2=1$ and combine (\ref{eq:Utheta-main-estimate}) and (\ref{eq:Utheta-weak-estimate}) with these powers:
% \[
% |\tilde U_{T_\theta}| = |\tilde U_{T_\theta}|^{2\beta_1}|\tilde U_{T_\theta}|^{\beta_2} \gtrsim \delta^{\varepsilon}r^{\beta_1} K^{-\beta_1}\d^{(2+t_2)\beta_1+\beta_2}\theta^{(2-t_2)\beta_1+\beta_2}|\T[T_\theta]|^{\beta_1}\lambda_{T_\theta}^{2\beta_1+2\beta_2} \lambda_{T_\theta, Y'}^{2\beta_1}. 
% \]

By the $(t_1, t_2,K)$-Katz--Tao axiom we have $|\T[T_\theta]| \lesssim M:=\min\{K (\theta/\d)^{t_1+t_2}, |\T|\}$, which we use to lower bound
\[
|\T[T_\theta]|^{1/2} = |\T[T_\theta]| /|\T[T_\theta]|^{1/2} \gtrsim |\T[T_\theta]| M^{-1/2}.
\]
Now we note that since $\lambda_{T_\theta} \ge  \lambda_{T_\theta, Y'} \ge \lambda_{T_\theta} - C r$ we can bound
\[
\sum_{T_\theta\in \T_\theta} |\T[T_\theta]| \lambda_{T_\theta} \lambda_{T_\theta, Y'} \ge \sum_{T_\theta\in \T_\theta} |\T[T_\theta]| \lambda_{T_\theta, Y'}^{2} \gtrsim |\T| \lambda^2_{Y'}
\]
where $\lambda_{Y'}$ is the average density of $Y'(T)$ over all $T\in \T$. Finally, we recall that $\sum_{\T} |\tilde Y(T)| \gtrsim \d^{2\varepsilon} \lambda \d^2 |\T|$ and so if we take $r \ll \d^{2\varepsilon} \lambda$ then we also get $\sum_{\T} |Y'(T)|\gtrsim \d^{2\varepsilon} \lambda \d^2 |\T|$. We conclude that
\[
|\tilde U| \gtrsim \sum_{T_\theta\in \T_\theta} |\tilde U_{T_\theta}| \gtrsim \delta^{\varepsilon}r^{1/2} K^{-1/2}\d^{(2+t_2)/2}\theta^{(2-t_2)/2}|\T| M^{-1/2}\lambda^2. 
\]
Now we observe
\begin{align*}
\theta^{(2-t_2)/2}|\T| &\max\{|\T|^{-1/2}, K^{-1/2} (\d/\theta)^{(t_1+t_2)/2}  \} \\&\ge \theta^{(2-t_2)/2}|\T| \left(|\T|^{-1/2}\right)^{1- \frac{2-t_2}{t_1+t_2}} \left(K^{-1/2}(\d/\theta)^{(t_1+t_2)/2}\right)^{\frac{2-t_2}{t_1+t_2}}\\
&= K^{-\frac{2-t_2}{2t_1+2t_2}}  \d^{\frac{2-t_2}{2}} |\T|^{\frac{2+t_1}{2t_1+2t_2}}
\end{align*}
and so putting everything together gives
\[
|\tilde U| \gtrsim \d^{C\varepsilon} K^{-\frac{2+t_1}{2t_1+2t_2}} \lambda^{5/2} \d^2 |\T|^{\frac{2+t_1}{2t_1+2t_2}}.
\]
\end{proof}

\section{High-low estimates in $\R^3$}\label{sec:high-low}

\subsection{Basic high-low inequality in $\R^3$}

Fix $\chi: \R^3\rightarrow [0,1]$ a compactly supported radial bump function such that $\chi(x) \ge 1/2$ for $|x|\le 1/2$ and $\chi(x) = 0$ for $|x|\ge 2$. For $w>0$ denote $\chi_w(x) = w^{-3} \chi(x/w)$ and define
$\eta_w = \chi_{w} * \chi_{w/2}$.

If $P$ and $L$ are sets of points and lines in $\R^3$ then we define the smooth incidence count as
\[
I(w; P,L) = w^2 \langle 1_P, \sum_{\ell \in L} 1_{\ell} * \eta_w\rangle
\]
where $1_P = \sum_{p\in P} 1_p$ is the sum of delta functions at points of $P$ and $1_\ell$ is the delta function on $\ell$. Explicitly, we have
\[
\langle 1_p, 1_\ell * \eta_w\rangle = \int_\ell \eta_w(x - p)dx.
\]
Define the normalized incidence function as
\[
B(w; P, L) = \frac{I(w, P, L)}{w^2 |P| |L|} = \langle g , f * \eta_w\rangle
\]
where $g = |P|^{-1} \sum_{p\in P} 1_p$ and $f = |L|^{-1}\sum_{\ell \in L} 1_{\ell}$ are the normalized indicator functions of the set of points and the set of lines.
In what follows, we will for convenience denote $B(w) = B(w; P, L)$ in cases when the sets $P, L$ are clear from the context.

We let $\mc A_{3,1}$ denote the set of lines in $\R^3$ which intersect the ball $B(0,1)$ and we fix a metric on $\mc A_{3,1}$. For a set of lines $L \subset \mc A_{3,1}$ and $u \le w \le 1$ we define $\M_{L}(u\times w\times 1)$ as the maximum over all $u\times w\times 1$ boxes $\Pi$ of $|L \cap \Pi|$ where $L\cap \Pi = \{\ell\in L:~ |\ell \cap \Pi| \ge 1/2\}$. 

\begin{prop}\label{prop:easy-relations-M}
    For any $L$ and $u\le u'\le 1$ and $w \le w' \le 1$ we have 
    \[
    \M_{L}(u'\times w'\times 1) \lesssim (u'/u)^2 (w'/w)^2 \M_L(u\times w\times 1).
    \]
    % and 
    % \[
    % \M_{L}(w\times w\times 1) \lesssim (w/u)^4 \M_L(u\times u\times 1).
    % \]
\end{prop}

\begin{proof}
    Given a collection of lines $L$, let $\Pi'$ be a $u'\times w'\times 1$ prism such that $|\Pi' \cap L| = \M_{L}(u'\times w'\times 1)$. Let $\Pi'\cap L = \{\ell_1, \ldots, \ell_m\}$, we have $|\Pi'\cap \ell_i| \ge 1/2$ for all $i$. Consider the two planes $P_1, P_2$ which contain the $u'\times w'$ faces of the prism $\Pi'$ and let $R'_1\subset P_1, R'_2\subset P_2$ be $4u' \times 4w'$ rectangles which are dilates of the corresponding faces of $\Pi'$. Then we have $\ell_i \cap P_j \in R'_j$ for $j=1,2$ and all $i$.  Now we cover each $R'_j$ by $\lesssim (u'/u)(w'/w)$ many disjoint $0.1u\times 0.1 w$ rectangles $R_{j,t}$. By pigeonhole principle, there is a pair of rectangles $R_{1, t_1}$, $R_{2,t_2}$ such that the number of indices $i$ so that $\ell_i \cap P_j \in R_{j, t_j}$ is lower bounded by $c (u/u')^2(w/w')^2 m$. One can check that the set of such lines is covered by some $u\times w\times 1$ prism $\Pi$. We conclude that $\M_L(u\times w\times 1) \gtrsim (u/u')^2(w/w')^2 m$, as desired.
\end{proof}

The next lemma is the simplest version of the high-low inequality in $\R^3$. This is essentially due to Guth--Solomon--Wang \cite{guth2019incidence} who used it to prove incidence upper bounds for well-spaced tubes in $\R^3$. Later on, we will prove two inequalities which improve on this bound under certain restrictions on the set of lines $L$. Both of these improvements will be crucial ingredients in the proof of Theorem \ref{thm:main}.

\begin{lemma}\label{lem:easy-high-low}
For any $\varepsilon>0$, $\delta >0$ and any $P \subset [-1,1]^3$, $L \subset \mc A_{3,1}$ we have 
\begin{equation}\label{eq:high-low-easy}
|B(\delta) - B(2\delta)|^2 \lesssim_\varepsilon \delta^{-6-\varepsilon} \frac{\M_P(\delta)}{|P|} \frac{\M_L(\delta\times \delta\times 1)}{|L|}.    
\end{equation}
\end{lemma}

\begin{proof}
    Let $g$ and $f$ be the normalized indicator functions of $P$ and $L$ as defined above. 
    By definition, we have 
    \begin{align*}
    B(\d) - B(2\d) = \langle g* \eta_\d, f\rangle - \langle g* \eta_{2\d}, f\rangle =  \langle g * \chi_\d, f * (\chi_{\d/2} - \chi_{2\d}) \rangle 
    \end{align*}
    Since $P \subset [-1,1]^3$ we may restrict the domain of integration to $[-2,2]^3$. 
    By the Cauchy-Schwarz inequality we then get
    \[
    |B(\d) - B(2\d)| \le \|g * \chi_\d\|_{L^2([-2,2]^3)} \|f * (\chi_{\d/2} - \chi_{2\d})\|_{L^2([-2,2]^3)}. 
    \]
    We can estimate
    \[
    \|g * \chi_\d\|^2_{L^2([-2,2]^3)} \le \|g * \chi_\d\|_{1} \|g * \chi_\d\|_{\infty} = \|g*\chi_\d\|_\infty\lesssim  \frac{\d^{-3}\M_P(\d)}{|P|}
    \]
    where we recall $\int_{\R^3} \chi_\d = 1$ and $g *\chi_\d(x)= |P|^{-1}\sum_{p\in P}\chi_\d(x-p)$.

    Letting $\psi = \chi_{\d/2} - \chi_{2\d}$, we now focus on $\|f * \psi\|^2_{L^2([-2,2]^3)}$. For a large enough constant $C>0$ we have
    \[
    \|f * \psi\|^2_{L^2([-2,2]^3)} \lesssim \| [\chi_C \cdot f] *\psi \|_2^2.
    \]
    Since $\chi$ is a fixed smooth function, we have the Fourier decay bound
    \[
    |\hat \psi(\xi)| \lesssim_d (1 + \d|\xi|)^{-d}
    \]
    for any $d\ge 1$. On the other hand, note that $\hat \chi_{w}(0) = 1$ for all $w$ and $\hat\chi$ is a smooth radially symmetric function. So the Taylor series of $\hat\psi = \hat\chi_{\d/2} - \hat\chi_{2\d}$ starts at degree 2, and so we have
    \[
    |\hat \psi(\xi)| \lesssim (\d |\xi|)^2.
    \]
    We conclude that
    \begin{equation}\label{eq:psi}
    |\hat \psi(\xi)| \lesssim_d (\d|\xi|)^2 (1 + \d|\xi|)^{-d} 
    \end{equation}
    for any fixed $d\ge 1$. %For large $\xi$ this follows from smoothness of $\chi$ and for small $\xi$ it is Taylor approximation at 0 (note that $\hat\psi(0) = 0$ and $\hat \psi$ is smooth and radially symmetric).

    Thus, by Plancherel's theorem, we need to estimate
    \[
    W(w, L) = \int_{|\xi| \sim w^{-1}} |\widehat{F}(\xi)|^2  d\xi
    , \quad\text{where}\quad F=\chi_C\cdot f, \]
    for all $w$. Note that $F$ really is a distribution, not a function, but the quantity $W(w, L)$ still has a well-defined meaning. Consider a $w$-net $\Theta \subset S^2$ and decompose $L = \bigsqcup_{\theta \in \Theta} L_{\theta}$ where $L_\theta$ consists of lines $\ell$ with direction $\theta(\ell)$ (arbitrarily choosing between $\pm\theta(\ell)$) satisfying $d(\theta(\ell), \theta) \le w$. Let $f_\theta = |L|^{-1} \sum_{\ell \in L_\theta} 1_{\ell}$ and $F_\theta = \chi_C\cdot f_\theta$.
    
    For each $\theta\in \Theta$ we can upper bound
    \begin{align}
    \int_{|\xi| \sim w^{-1}} |\widehat{F_\theta}(\xi)|^2 d\xi&\lesssim \int_{|\xi| \sim w^{-1}} |\widehat{F_\theta}(\xi)|^2 |\widehat\chi_{w/C}|^2  d\xi \le \int_{\R^3} |\widehat{F_\theta}(\xi)|^2 | \widehat\chi_{w/C}|^2 d\xi = \nonumber \\
    =\| F_\theta * \chi_{w/C} \|^2_{2} 
    &\le \| F_\theta * \chi_{w/C} \|_{1} \| F_\theta * \chi_{w/C} \|_{\infty} \lesssim w^{-2} |L_\theta| \M_L(w\times w\times 1) |L|^{-2}.\label{eq:bound-Ftheta}
    \end{align}
    
    For $|\xi|\sim w^{-1}$, let $\theta(\xi)=\frac{\xi}{|\xi|}$. By the rapid decay of $\widehat \chi_C$, the function $\widehat{\chi_C \cdot 1_\ell}$ is essentially supported on the $1$-neigborhood of the plane $\theta^{\perp} \subset \R^3$. This implies that we have
    \[
    \widehat{F}(\xi) = \sum_{\theta \in \Theta:~ d(\theta, \theta(\xi)^{\perp}) \le K w} \widehat{F_\theta}(\xi) + O_d(w^{-2} K^{-d})
    \]
    for any fixed $d \ge 1$. Since there are at most $C K w^{-1}$ many $\theta \in \Theta$ in a strip of width $K w$ on $S^2$, for any fixed $\xi$ we have by Cauchy--Schwarz:
    \[
    |\sum_{\theta \in \Theta:~ d(\theta, \theta(\xi)^{\perp}) \le K w} \widehat{F_\theta}(\xi)|^2 \le C K w^{-1} \sum_{\theta \in \Theta:~ d(\theta, \theta(\xi)^{\perp}) \le K w} |\widehat{F_\theta}(\xi)|^2
    \]
    and so
    \begin{align*}
    |\widehat{F}(\xi)|^2 \lesssim Kw^{-1} \sum_{\theta \in \Theta:~ d(\theta, \theta(\xi)^{\perp}) \le Kw} |\widehat{F_\theta}(\xi)|^2 + O_d(w^{-4} K^{-d}) \le Kw^{-1} \sum_{\theta \in \Theta} |\widehat{F_\theta}(\xi)|^2 + O_d(w^{-4} K^{-d}).     
    \end{align*}
    So by integrating this over all $\xi$ and applying (\ref{eq:bound-Ftheta}) this gives
    \begin{align*}
    \int_{|\xi| \sim w^{-1}} |\widehat{F}(\xi)|^2 &\lesssim Kw^{-1} \sum_{\theta} w^{-2}|L_\theta| \M_L(w\times w\times 1) |L|^{-2} + O_d(w^{-7} K^{-d}) \\
    &\lesssim K w^{-3} \frac{\M_L(w\times w\times 1)}{|L|}  + O_d(w^{-7} K^{-d}). 
    \end{align*}

    Thus, we obtain the following estimate for any $d\ge 1$ and any $K \ge 1$:
    \[
    W(w, L) \lesssim Kw^{-3} \frac{\M_L(w\times w\times 1)}{|L|}  + O_d(w^{-7} K^{-d}).
    \]
    By Proposition \ref{prop:easy-relations-M} for dyadic values of $w\in(K^{-1}\d,C)$ we have $\M_L(w\times w\times 1)\lesssim (1+(w/\d)^{4})\M_L(\d\times \d\times 1)$. So by (\ref{eq:psi}) we can estimate $\| F *\psi \|_2^2$ by:
    \begin{align*}
        \| F *\psi \|_2^2& \lesssim \int_{|\xi|\lesssim 1}|\widehat{F}\widehat{\psi}|^2+\sum_{K^{-1}\d<w\lesssim 1} W(w,L) (\d w^{-1})^4+\int_{|\xi|>K\d^{-1}}|\widehat{F}\widehat{\psi}|^2 \\
        &\lesssim \int_{|\xi|\lesssim 1}|\widehat{F}\widehat{\psi}|^2+\sum_{K^{-1}\d<w\lesssim 1} \big(K^2w^{-3} \frac{\M_L(w\times w\times 1)}{|L|}\big) (\d w^{-1})^4+O_d(\d^{-7}K^{-d}) \\
        &\lesssim \int_{|\xi|\lesssim 1}|\widehat{F}\widehat{\psi}|^2+\sum_{K^{-1}\d<w\lesssim 1} \big(K^6w \d^{-4} \frac{\M_L(\d\times \d\times 1)}{|L|}\big) (\d w^{-1})^4+O_d(\d^{-7}K^{-d}) \\
        &\lesssim \int_{|\xi|\lesssim 1}|\widehat{F} \widehat{\psi}|^2+K^7\d^{-3}\frac{\M_L(\d\times \d\times 1)}{|L|}+O_d(\d^{-7}K^{-d}) . 
    \end{align*}

The first term in the upper bound is controlled by
\[
\int_{|\xi|\lesssim 1}|\widehat{F} \widehat{\psi}|^2 \lesssim \delta^{4} \int _{|\xi|\lesssim 1}|\widehat{F}|^2 \lesssim \d^4 \int F^2*\chi_c \lesssim \delta^4
\]
which is negligible (in particular it is smaller than the second term), so altogether we have
\[ 
\| f *\psi \|_{L^2([-2,2]^3)}^2\lesssim K^7 \d^{-3}\frac{\M_L(\d\times \d\times 1)}{|L|}+O_d(\d^{-7}K^{-d}). \]
By choosing $K=\d^{-\varepsilon/10}$ and $d = 10\varepsilon^{-1}$ we conclude the proof. 
\end{proof}

\subsection{Improved high-low using two-ends decomposition on Fourier side.} 
%Define $\M_{L}(u\times w\times 1)$ to be the maximum number of lines $\ell \in L$ contained in a $1\times u\times w$ prism $\Pi$ (in the sense that $|\ell \cap \Pi| \ge 1/2$). 
For a set of lines $L$ we let $\theta(L) \subset S^2$ be the set of directions spanned by lines in $L$ and for $w>0$ we write $|\theta(L)|_w$ to be the $w$-covering number of the set of directions.

In the next theorem, we prove a refined version of the high-low inequality which takes into account information about the set of directions of $L$ and the concentration of $L$ in prisms.

\begin{theorem}\label{thm:refined_high-low}
    Let $P \subset [-1,1]^3$ and $L \subset \mc A_{3,1}$ be finite collections of points and lines and let $\d >0$. Then we have 
    \begin{align}\label{eq:refined-high-low}
    |B(\d)-B(2\d)|^2 \lesssim_\varepsilon \d^{-6-\varepsilon} \frac{\M_P(\d)}{|P|} \max_{u \in (\delta, 1)}\left( \min\{|\theta(L)|_\d^{1/2} \d, u\} u \frac{\M_{L}(\delta\times \delta/u \times 1)}{|L|} \right).
    \end{align}
    %|B(\d) - B(2\d)|^2 \lesssim_\epsilon \d^{-6-\epsilon} \frac{\M_P(\d)}{|P|} \max_{u \in (\d,1)} &\min\big\{  u^2 \frac{\M_L(1\times \d/u \times \d)}{|L|},\\
     %& \d |\theta(L)|_\d^{1/2} u^{1/2} \frac{\M_L(1\times \d \times \d)^{1/2} \M_L(1\times \d/u \times \d)^{1/2} }{|L|} \big\}.\nonumber
\end{theorem}

To prove this estimate, we run the proof of Lemma \ref{lem:easy-high-low} up until the point where we upper bound the number of angles $\theta$ is a strip of width $Kw$ by the trivial upper bound $CKw^{-1}$. At this stage, if we know that $|\Theta| \ll w^{-2}$, we can apply the two-ends decomposition (Lemma \ref{lem:two-ends}) to obtain that multiplicity is upper bounded by $|\Theta|^{1/2}\ll w^{-1}$, apart from a few shorter segments where multiplicity can still be large. To deal with those segments, we can rescale the picture and repeat the argument.

Before the proof, let us formulate a useful corollary which gives sufficient conditions to improve over the basic high-low estimate from Lemma \ref{lem:easy-high-low}.

\begin{cor}\label{cor:refined-high-low}
    Let $P \subset [-1,1]^3$ and $L \subset \mc A_{3,1}$ be finite collections of points and lines and let $\d >0$. Suppose that for some $\kappa\in [0,1]$, $\nu>0$ and $M\ge 1$ we have $|\theta(L)|_\d \le \nu \d^{-2}$ and $\M_L(\d\times \d/u \times 1) \le u^{-2+\kappa} M$ for any $u \in (\d,1]$. Then we have
    \[
    |B(\d) - B(2\d)|^2 \lesssim_\varepsilon \nu^{\kappa/4}\d^{-6-\varepsilon} \frac{\M_P(\d)}{|P|}  \frac{M}{|L|}
    \]
\end{cor}

\begin{proof}
    Note that taking $u=1$ implies that $\M_L(\d\times \d\times 1) \le M$.
    We use (\ref{eq:refined-high-low}) and split into two ranges: if $u \in [\nu^{1/4}, 1]$ then 
    \[
     \min\{|\theta(L)|_\d^{1/2} \d, u\} u \frac{\M_{L}(\delta\times \delta/u \times 1)}{|L|} \le \nu^{1/2} u \frac{\M_{L}(\delta\times \delta/u \times 1)}{|L|} 
    \]
    we have $\M_L(\delta\times \delta/u \times 1) \lesssim u^{-2} \M_L(\d\times \d\times 1)$ (by Proposition \ref{prop:easy-relations-M}) and so the right hand side is at most $\nu^{1/4} \frac{\M_{L}(\delta\times \delta \times 1)}{|L|} \lesssim \nu^{\kappa/4} \frac{M}{|L|}$. For $u \in [\d, \nu^{1/4}]$ we have 
    \[
    \min\{|\theta(L)|_\d^{1/2} \d, u\} u \frac{\M_{L}(\delta\times \delta/u \times 1)}{|L|} \le u^2 \frac{ u^{-2+\kappa} M}{|L|} \le \nu^{\kappa/4} \frac{M}{|L|}.
    \]
\end{proof}

%So whenever the line set $L$ spans a bit fewer than maximum number of directions on scale $\d$ and satisfies a concentration estimate, we obtain an improvement over (\ref{eq:high-low-easy}). 

\begin{proof}[Proof of Theorem \ref{thm:refined_high-low}] 
    For $\varepsilon >0$ take $K = \delta^{-\varepsilon/100}$ and $\Delta = \delta^{\varepsilon/100}$ and $d = 10^4\varepsilon^{-1}$. 
    
    Let $g = \frac{1}{|P|}1_P$ and $f = \frac{1}{|L|} \sum_{\ell\in L} 1_\ell$ be the normalized indicator functions of $P$ and $L$, respectively. By repeating the proof of Lemma \ref{lem:easy-high-low}, we let $F = \chi_C \cdot f$ and $\psi = \chi_{\d/2} - \chi_{2\d}$ and estimate
    \begin{equation}\label{eq:Bdelta-CS-2}
    |B(\d) - B(2\d)| \le \|g * \chi_\d\|_{L^2([-2,2]^3)} \|f * \psi\|_{L^2([-2,2]^3)} \lesssim \left( \d^{-3} \M_P(\d) |P|^{-1} \right)^{1/2} \| F * \psi \|_2    
    \end{equation}
    and reduce the problem to understanding the dyadic weights
    \[
    W(w, L) = \int_{|\xi| \sim w^{-1}} |\widehat F|^2
    \]
    for an arbitrary $w \in (K^{-1}\d,1)$. We split $L = \bigsqcup_{\Theta} L_\theta$ where $\Theta=\Theta_w$ is a $w$-separated set of directions whose $w$-neighborhood covers $\theta(L)$ and $|\Theta| \sim |\theta(L)|_w$. Let $f_\theta = |L|^{-1} \sum_{L_\theta} 1_\ell$ and $F_\theta = \chi_C\cdot f_\theta$ be the corresponding functions. As in the proof of Lemma \ref{lem:easy-high-low} we have the estimate
    \begin{equation}\label{eq:upper-bound-Ftheta}
    \int_{|\xi|\sim w^{-1}} |\widehat F_\theta|^2 \lesssim w^{-2} |L_\theta| \M_L(w\times w\times 1) |L|^{-2}    
    \end{equation}
    and we have an approximation
    \[
    \widehat{F}(\xi) = \sum_{\theta \in \Theta:~ d(\theta, \theta(\xi)^{\perp}) \le K w} \widehat{F_\theta}(\xi) + O_d(w^{-2} K^{-d})
    \]
    At this point in the proof of Lemma \ref{lem:easy-high-low}, we used the fact that the $Kw$-neighborhood of a great circle contains at most $\sim K w^{-1}$ directions $\theta\in \Theta$. Here we observe that if $|\Theta|_w \ll w^{-2}$ then for most $\xi$ we have much fewer such directions. To make this precise, we iteratively apply Corollary \ref{cor:two-ends}.

    For $a\le b\le 2\pi$ we define a $a\times b$ rectangle on the sphere $S^2$ to be the $a$-neighborhood of a great circle segment of length $b$. Let 
    Given an  rectangle $U\subset S^2$, we can define a smooth bump function $\rho_U$ supported on the rectangular set $\{|\xi| \in (w^{-1}/4, 2 w^{-1}), ~ \theta(\xi) \in 1.1U\}$ and equals 1 on the set $\{|\xi| \in (w^{-1}/2,  w^{-1}), ~ \theta(\xi) \in U\}$. 
    
    Denote $\d_0 = Kw$. For $\theta \in \Theta$ let $U = U_\theta$ be the $\d_0$-neighborhood of $\theta^{\perp}$ on $S^2$ and let $\mathbb U_0 = \{U_\theta, ~\theta \in \Theta\}$. Fix some $\Delta > 0$ and let $m = C \frac{\log (2/\d_0)}{\log (2/\Delta)}$. For $j=1, \ldots$ we are going to construct a collection of essentially distinct $ 20^j \d_0 \times \Delta^j $-rectangles $\mathbb U_j$ on $S^2$ inductively as follows. Given a collection of rectangles $\mathbb U_{j}$ with $j\ge 0$, we do the following:
    \begin{itemize}
        \item Consider a collection of balls $\{B_{j,i}\}$ of radius $C\Delta^j$ on $S^2$ which are $O(1)$-overlapping and every $3\Delta^j$-ball is contained in at least one $B_{j,i}$, and set $\mathbb U_j[B_i] = \{U\in \mathbb U_j:~ U \subset B_i\}$,
        \item Apply Corollary \ref{cor:two-ends} (with $\Delta/1.1$ in place of $\Delta$) to each $\mathbb U_j[B]$ to obtain a set of essentially distinct $10\cdot 20^j\d_0\times \Delta^{j+1}/1.1$ rectangles $\mathbb U_{j+1, B}$ inside $B$. For each $U \in \mathbb U_j[B]$ we thus get some $\mathbb U(U, B) \subset \mathbb U_{j+1,B}$ of size at most $m$ so that sets 
        \[
        1.5 U \setminus \bigcup_{U' \in \mathbb U(U, B)} U'
        \]
        are at most $C\Delta^{-2}|\mathbb U_j[B]|^{1/2}$-overlapping.
        \item We let $\mathbb U_{j+1}$ be a set of essentially distinct $20^{j+1} \d_0\times \Delta^{j+1}$ rectangles covering the union of sets $\mathbb U_{j+1, B}$ over all balls $B$ (here `essentially distinct' means that, say, $|U\cap U'| \le (1-c) |U|$ for any $U\neq U' \in \mathbb U_{j+1}$ and some constant $c>0$). Further we let $\mathbb U(U) \subset \mathbb U_{j+1}$ be the subset of rectangles $U \in \mathbb U_{j+1}$ which contain at least one rectangle in $\bigcup_B\mathbb U(U, B)$ where the union is over all balls $B$ containing $U$.
    \end{itemize}

    We run this process for $j=0, \ldots, j_0$ where $j_0$ is the minimum index such that $(\Delta/100)^j < \d_0$ holds. We obtain collections of rectangles $\mathbb U_j$, $j=0, \ldots, j_0$ such that for each $j< j_0$ and $U \in \mathbb U_j$ we have a selected subset $\mathbb U(U) \subset \mathbb U_{j+1}$ of size at most $m$. It follows from construction that for any $j< j_0$, the collection of sets $1.1U \setminus (\bigcup_{U'\in \mathbb U(U)} U')$ over $U\in \mathbb U_j$ is at most
    \begin{equation}\label{eq:upper-bound-overlap}
    C\Delta^{-2} \max_{i} \{ |\mathbb U_j[B_{j,i}] |^{1/2}\} \le C\Delta^{-2} |\mathbb U_j|^{1/2} \le C \Delta^{-2} (m^j|\Theta|)^{1/2}.
    \end{equation}
    (This bound is quite wasteful but it will suffice for us.)
    For $j = j_0$ let us put for convenience $\mathbb U(U) = \emptyset$. 

    Now we use this data to decompose functions $\widehat{F}_\theta$ into pieces with good overlapping properties. For $U \in \mathbb U_j$ we define a function $F_U$ essentially supported on the region $\overline{U}=\{|\xi| \sim w^{-1}, \theta(\xi) \in U\}$ inductively as follows. For $j=0$ we define $G_U = \rho_U \cdot \widehat{F}_\theta$ for each $U = U_\theta \in \mathbb U_0$. Given $U \in \mathbb U_j$ for some $j\ge 0$ let us write $\mathbb U(U) = \{U_1, \ldots, U_k\} \subset \mathbb U_{j+1}$ and define functions   
    \begin{align}
    F_U &= G_U \cdot \prod_{t=1}^k (1-\rho_{U_{t}}),\label{eq:rho-F}\\
    G_U^{(t)} &= G_U \cdot \rho_{U_t} \prod_{t'=1}^{t-1} (1-\rho_{U_{t}}), \quad t=1, \ldots, k \label{eq:rho-G}
    \end{align}
    so that $G_U = F_U + G^{(1)}_U+\ldots + G^{(k)}_{U}$. We then define for each $U' \in \mathbb U_{j+1}$:
    \[
    G_{U'} = \sum_{U\in \mathbb U_j,~ t:~ U_t = U'} G_U^{(t)}
    \]
    and proceed with the construction to the next index $j$.  
    By the definition of $\rho_U$, the function $F_U$ is then supported on the set
    \[
    \{ \xi:~ |\xi|\sim w^{-1}, ~ \theta(\xi) \in  1.1 U \setminus \bigcup_{U' \in \mathbb U(U)} U'\}.
    \]
    So by (\ref{eq:upper-bound-overlap}), the supports of functions $F_U$, $U\in \mathbb U_j$ for $j<j_0$ are at most $C\Delta^{-2}m^j|\Theta|^{1/2}$ overlapping. For any $j\in \{0, \ldots, j_0\}$ the supports of $F_U$, $U\in \mathbb U_{j_0}$, are trivially at most $C (\Delta^{j} / 20^{j}\d_0)$-overlapping (since this is a collection of essentially distinct $20^{j} \d_0 \times \Delta^{j}$-rectangles). By the choice of $j_0$, we have  $C (\Delta^{j_0} / 20^{j_0}\d_0) \le C 100^{j_0}$. Thus, for any $j\in \{0, \ldots, j_0\}$ the supports of $F_U$, $U\in \mathbb U_j$ are $C\min(\Delta^{-2}(100m)^j|\Theta|^{1/2}, \Delta^{j} / 20^{j}\d_0)$-overlapping (note that for $j=j_0$ the second term dominates).
    
    % \[
    % C\min\{\Delta^{-2}|\Theta|^{1/2}, \Delta^j \d_0^{-1}\}
    % \]
    % overlapping. Here we use a (very crude) estimate $|\mathbb U_{j, B}| \lesssim |\Theta|$ which holds for all $j\ge 0$ and all $B$.%\footnote{A more careful analysis of set of directions $\Theta$ leads to a better estimate in (\ref{eq:refined-high-low}). However this improvement turns out not essential for our application and so we opted for a simpler argument.}. 
    By this construction, we obtain a decomposition
    \[
    \widehat{F} = \sum \rho_{U_\theta}\widehat{F}_\theta + \sum (1-\rho_{U_\theta})\widehat{F}_\theta,
    \]
    \[
    \sum \rho_{U_\theta}\widehat{F}_\theta= \sum_{j=0}^{j_0} \sum_{U \in \mathbb U_j} F_U.
    \]
    We can estimate 
    \[
    \int_{|\xi| \sim w^{-1}} \left | \sum (1-\rho_{U_\theta})\widehat{F}_\theta\right|^2 \lesssim_d w^{-3} K^{-d} 
    \]
    by the rapid decay of $1_{|\xi| \sim w^{-1}}\widehat{F}_\theta $ outside $U_\theta$. The remainder sum $\sum \rho_{U_\theta}\widehat{F}_\theta$ can be estimated as
    \begin{align}\label{eq:upper-bound-F-U}
    \int_{|\xi| \sim w^{-1}} |\sum \rho_{U_\theta}\widehat{F}_\theta|^2 \le (j_0+1) \sum_{j=0}^{j_0}  \int_{|\xi| \sim w^{-1}} |\sum_{\mathbb U_j}F_U|^2 \nonumber\\
    \lesssim j_0 \sum_j C\min(\Delta^{-2}(100m)^j|\Theta|^{1/2}, \Delta^{j} / 20^{j}\d_0) \sum_{\mathbb U_j} \int_{|\xi| \sim w^{-1}} |F_U|^2,
    \end{align}
    where we used the overlap property of supports of functions $F_U$. 
    It remains to estimate the $L_2$-norm of the function $F_U$. 
    Let $L_U$ be the set of lines `contributing' to $F_U$, namely $L_U$ is the union of the line sets $L_\theta$ over all $\theta$ such that there exists a chain 
    \[
    U_\theta = U_0, U_1, \ldots, U_j = U
    \]
    where $U_{i+1} \in \mathbb U(U_i)$ for every $i =0, \ldots, j-1$. Observe that since $|\mathbb U(U)|\le m$ we have
    \begin{equation}\label{eq:LU-dont-overlap}
    \sum_{\mathbb U_j} |L_U| \le m^j |L|.    
    \end{equation}
    We can write 
    \[
    F_U = |L|^{-1}\sum_{\ell \in L_U} \widehat{(1_\ell \cdot \chi_C)} \cdot \rho_{U, \ell},
    \]
    where $\rho_{U, \ell}$ is a smooth $1$-bounded function supported on $\overline U = \{\xi: |\xi|\sim w^{-1}, \theta(\xi) \in 1.1U\}$ and composed of products and sums of various functions $\rho_{U'}$ over some $U' \in \mathbb U_{j'}$ with $j' \in \{1, \ldots, j+1\}$, see (\ref{eq:rho-F}), (\ref{eq:rho-G}). This implies that $\widecheck{\rho}_{U, \ell}$ is a linear combination of convolutions of various $\widecheck{\rho}_{U'}$. For $U' \in \mathbb U_j$ recall that $\rho_{U'}$ was chosen to be a smooth bump function on $\overline{U'}$. Since $U'$ is a spherical $20^j\d_0\times \Delta^j$ rectangle, it follows that $\overline{U'}$ is approximately a $w^{-1} 20^{j}\d_0 \times w^{-1} \Delta^j  \times w^{-1}$ box. So if we let $\overline{U'}^\vee$ be the $w/(20^j \d_0) \times w/\Delta^{j} \times w$ box which is dual to $\overline{U'}$ and is centrally symmetric around the origin, then we have the decay estimate 
    \[
    |\widecheck{\rho}_{U'}(x)| \lesssim_d K^{-d} |\overline{U'}|, \quad x\in \R^3\setminus K \overline{U'}^\vee.
    \]
    Now note that if $\rho_{U'}$ appears in the expression for $\rho_{U, \ell}$ then rectangles $U$ and $U'$ have approximately aligned major axes (by the last part of Corollary \ref{cor:two-ends}) and so the union of boxes $\overline{U'}^\vee$ over all such $U'$ is contained in the symmetric around the origin $w/\d_0 \times w/\Delta^{j+1} \times w$ box $\tilde U^\vee$ (where the axes of $\tilde U^\vee$ coincide with the axes of $\overline{U}$). So we conclude that 
    \[
    |\widecheck{\rho}_{U, \ell}(x)| \lesssim_{d, m} K^{-d}|\overline{U}|, \quad x\in \R^3\setminus K \tilde{U}^\vee.
    \]
    By similar analysis we can show that $\|\widecheck{\rho}_{U, \ell}\|_1 \lesssim_m 1$ and $\|\widecheck{\rho}_{U, \ell}\|_\infty \lesssim |\overline{U}|$. 
    Using this we estimate
    \begin{align*}
    \int_{|\xi|\sim w^{-1}} |F_U|^2 \lesssim \int_{\R^3} |F_U|^2 |\widehat{\chi}_{w/C}|^2 = \|\widecheck{F}_U * \chi_{w/C}\|^2_{2} \le \|\widecheck{F}_U * \chi_{w/C}\|_1 \|\widecheck{F}_U * \chi_{w/C}\|_\infty
    \end{align*}
    we have 
    \[
    \widecheck{F}_U = |L|^{-1}\sum_{\ell \in L_U} (1_\ell \cdot \chi_C) * \widecheck{\rho}_{U, \ell}
    \]
    and so the $L_1$-norm is upper bounded as
    \[
    \|\widecheck{F}_U * \chi_{w/C}\|_1 \le C|L|^{-1} \sum_{\ell \in L_U} \|\widecheck{\rho}_{U, \ell}\|_1 \lesssim \frac{|L_U|}{|L|}.
    \]
    For the $L_\infty$-norm we estimate
    \[
    |\widecheck{F}_U * \chi_{w/C}(x)| \lesssim \frac{1}{|L|} \sum_{\ell \in L_U} 1_{x \in (\ell \cap B^3(0,C) + K\tilde U^\vee )} \|\widecheck{\rho}_{U, \ell}\|_\infty   + O(|\overline{U}|K^{-d})
    \]
    since the direction of $\ell \in L_U$ is aligned with the longest axis of the box $K\tilde U^\vee $  (which by definition has length $w/\d_0=K$), we conclude that the set $(\ell \cap B^3(0,C) + K\tilde U^\vee )$ is contained in $C'K^2\tilde U^\vee$ for some constant $C'$. So the number of lines $\ell \in L_U$ such that $x$ is contained in $C'K^2\tilde U^\vee$ is upper bounded by $\M_L(C'K^2w \times C'K^2 w/\Delta^{j+1} \times 1) \lesssim K^{8}\Delta^{-2}\M_L(w\times w/\Delta^{j}\times 1)$. Using this we obtain 
    \[
    \|\widecheck{F}_U * \chi_{w/C}\|_\infty \lesssim K^8 \frac{1}{|L|} \M_L(w\times w/\Delta^{j+1}\times 1) |\overline{U}|   + O(|\overline{U}|K^{-d}).
    \]
    Take $K = \delta^{-\varepsilon/100}$ and $\Delta = \delta^{\varepsilon/100}$ and $d = 10^4\varepsilon^{-1}$. Note that then we have $m, j_0 \lesssim \varepsilon^{-1}$ and so in particular, terms of the shape $100^{j_0}, m^j$ are all $\lesssim_\varepsilon 1$. Recall that $|\overline{U}| \sim w^{-3} (20\Delta)^j \d_0 \lesssim_\varepsilon \Delta^j w^{-2}$. 
    Thus, we obtain
    \[
    \int_{|\xi|\sim w^{-1}} |F_U|^2 \lesssim_\varepsilon \delta^{-\varepsilon/3} \Delta^j w^{-2} \frac{|L_U|\M_L(w \times w /\Delta^{j} \times 1)} {|L|^{2}}
    \]
    (cf. (\ref{eq:upper-bound-Ftheta})). So plugging this into (\ref{eq:upper-bound-F-U}) and using (\ref{eq:LU-dont-overlap}) gives
    \begin{equation}\label{eq:WwL}
    W(w, L) = \int_{|\xi| \sim w^{-1}} |\widehat F|^2 \lesssim_\varepsilon  \delta^{-\varepsilon/2} \max_{u \in (w,1)} \min\{|\Theta|^{1/2}, u w^{-1}\} u w^{-2} \frac{\M_{L}(w\times w/u \times 1)}{|L|}
    \end{equation}

    Now we use the fact that $w^{-4} \M_L(w\times w/u\times 1) \lesssim K^4 \delta^{-4} \M_L(\delta \times \delta/u \times 1)$ (by Proposition \ref{prop:easy-relations-M})  we get that
    \[
    W(w, L) \lesssim_\varepsilon 
    \delta^{-2\varepsilon/3} \d^{-4} \frac{\M_{L}(\d\times \d/u \times 1)}{|L|}     \max_{u \in (w,1)} \min\{|\Theta|^{1/2}, u w^{-1}\} u w^{2}.
    \]
    Recall that $|\Theta| = |\Theta_w| \sim |\theta(L)|_w \lesssim K^2 |\theta(L)|_\d$.
    So by the decay estimates (\ref{eq:psi}) on $\widehat \psi$ we conclude that 
    \begin{align*}
    \|f * \psi\|^2_{L^2([-2,2]^3)} &\lesssim \sum_w  W(w, L) (\d w^{-1})^2 (1 + \d w^{-1})^{-d} \\
    &\lesssim_\varepsilon \sum_w \delta^{-2\varepsilon/3} \d^{-2} \frac{\M_{L}(\d\times \d/u \times 1)}{|L|}     \max_{u \in (w,1)} \min\{|\theta(L)|_\d^{1/2}, u w^{-1}\} u (1 + \d w^{-1})^{-d} \\
    &\lesssim_\varepsilon \delta^{-\varepsilon} \max_{u \in (\delta, 1)} \left(\min\{|\theta(L)|_\d^{1/2}, u \delta^{-1}\} u \delta^{-2} \frac{\M_{L}(\delta\times \delta/u \times 1)}{|L|} \right) + O(\d^{-7} K^{-d}).
    \end{align*}
    % the right hand side of (\ref{eq:WwL}) is maximized at $w \sim \delta$. We we get
    % \[
    % \|f * \psi\|^2_{L^2([-2,2]^3)} \lesssim \delta^{-\epsilon} \max_{u \in (\delta, 1)} \left(\min\{|\theta(L)|_\d^{1/2}, u \delta^{-1}\} u \delta^{-2} \frac{\M_{L}(\delta\times \delta/u \times 1)}{|L|} \right) + O(\d^{-7} K^{-d})
    % \]
    % leading to the estimate
    so plugging this back into (\ref{eq:Bdelta-CS-2}) gives
    \[
    |B(\d)-B(2\d)|^2 \lesssim_\varepsilon \d^{-3-\varepsilon} \frac{\M_P(\d)}{|P|} \max_{u \in (\delta, 1)}\left( \min\{|\theta(L)|_\d^{1/2}, u \delta^{-1}\} u \delta^{-2} \frac{\M_{L}(\delta\times \delta/u \times 1)}{|L|} \right)
    \]
    which concludes the proof.
\end{proof}

\subsection{Improved high-low using rescaling and hairbrush estimate.}
The next lemma provides an alternative approach to improving Lemma \ref{lem:easy-high-low}. The idea is to split $\d$-tubes into $\d\times \d\times \d^{1/2}$ segments and to apply the simple high-low (Lemma \ref{lem:easy-high-low}) to the short tubes inside $\d^{1/2}$ boxes. The crux of this approach is estimating how much can the $\d\times \d\times \d^{1/2}$ segments cluster inside the $\d^{1/2}$ boxes. By decomposing the set of $\d$-tubes into a union over thick $\d^{1/2}$ tubes this can be restated as a lower bound on the volume of some shading of tubes inside a fixed thick tube. Assuming the original set of tubes is $(t_1, t_2, C)$-Katz--Tao we can use Lemma \ref{lem:space-brush} to estimate this volume.

We say that a set of lines $L$ is $C_0$-uniform at scale $\d$ if for every $\ell \in L$ the number of lines $\ell' \in L$ such that $\ell' \cap B(0,1) \subset T_\ell(\d)$ is at least $\M_L(\d\times \d\times 1) / C_0$. Later on, we will introduce a more systematic approach to uniformity but this will suffice for now.  

\begin{theorem}\label{thm:small-box-high-low}
    Let $P \subset [0,1]^3$ and $L \subset \mc A_{3,1}$ be  finite collections of points and lines. Suppose that $\M_P(\d) \le A \d^3 |P|$ for some $A\ge 1$.

    Suppose that $L$ is $C_0$-uniform at scale $\d$ and that $L$ defines a $(t_1, t_2, K)$-Katz-Tao set of $\d$-tubes. Denote $\alpha = \frac{t_1+2}{2t_1+2t_2}$.
    Then we have
    \begin{align}
    |B(\d/2) - B(\d)|^{9/2} \lesssim_\varepsilon
    C_0^{O(1)}\d^{-\varepsilon}\d^{-5/2} K^{\alpha}A^{7/2}\frac{\M_L(\d^{1/2} \times \d^{1/2} \times 1)^{1-\alpha}\M_L(\d \times \d \times 1)^{\alpha}}{|L|}
\label{eq:second-high-low-improvement}    
    \end{align}
    %\lesssim_\varepsilon C_0K^{1/2}\d^{-5/4-\varepsilon} &A^{3/2}\frac{\M_L(\d^{1/2}\times\d^{1/2}\times 1)^{\frac{1-\alpha}{2}}  \M_L(\d \times \d\times 1)^{\frac{\alpha}{2}}}{|L|^{\frac12}}.
    %\blue{Should we worry about getting the right power of $K$?}
\end{theorem}

We will apply this estimate in the situations of the following type. Suppose that $P, L$ are $\d$-separated sets of points and lines (say, coming from a point-line configuration $X$ with $d(X) \ge \d$) and that we have $|P|, |L| \approx_\varepsilon \d^{-3}$. Say that we know from Proposition \ref{prop:plane-reduction} that $L$ defines a $(1+\gamma, 2-\gamma, \d^{-\varepsilon})$-Katz--Tao set of $\d$-tubes. Moreover, let us suppose that $L$ is {\em well-spaced}, namely that lines in $L$ are well distributed on scale $\d^{1/2}$, i.e. $\M_L(\d^{1/2}\times \d^{1/2}\times 1) \lessapprox \d^2 |L| \approx \d^{-1}$ holds. Under these assumptions, we have $\alpha = \frac{1}{2}+\frac{\gamma}{6}$ and 
(\ref{eq:second-high-low-improvement}) gives 
\[
|B(\d/2) - B(\d)|^{9/2} \lessapprox_\varepsilon \d^{-5/2} \frac{\M_L(\d^{1/2}\times\d^{1/2}\times 1)^{1-\alpha}  }{|L|} \lessapprox \d^{-\frac{1}{2}+\alpha} = \d^{\frac{\gamma}{6}}
\]
and so we get a polynomial improvement over Lemma \ref{lem:easy-high-low} provided that $\gamma>0$. 

\begin{proof}[Proof of Theorem \ref{thm:small-box-high-low}]
    Using $C_0$-uniformity of $L$ we can find a collection essentially distinct $\d$-tubes $\T$ and a finitely overlapping covering $L = \bigcup_{T\in \T} L[T]$ so that $|\T| \lesssim C_0 \frac{|L|}{\M_L(\d\times \d\times 1)}$ and $1/C_0\lesssim \frac{|L[T]|}{\M_L(\d\times \d\times 1)} \lesssim 1$ for every $T \in \T$. Then $\T$ is a $(t_1, t_2, C_0K)$-Katz--Tao set of $\d$-tubes. 

    Denote $\Delta = \d^{1/2}$. 
    Decompose $P = \bigsqcup_{Q\in \mathcal Q} P[Q]$ where $\mathcal Q=\mathcal Q_\Delta$ is a collection of disjoint $\Delta$-boxes covering $P$. 
    Fix some parameter $\tau > 0$ and for each $\ell \in L$ consider the set $\mathcal Q(\ell)$ of $\Delta$-boxes $Q \in \mathcal Q$ such that
    \begin{equation}\label{eq:difference}
    \langle 1_\ell * (\eta_\d - \eta_{\d/2 }), 1_{P[Q]} \rangle \ge \tau \Delta|P|.
    \end{equation}
    Using the point-wise bound $|1_\ell *\eta_\delta| \lesssim \d^{-2}$, for any $Q\in \mc Q$ intersecting $T_{2\d}(\ell)$ we have crude bounds
    \[
    |\langle 1_\ell * \eta_\d, 1_{P[Q]}\rangle|+|\langle 1_\ell * \eta_{\d/2}, 1_{P[Q]}\rangle|  \lesssim \d^{-2} |T_{2\d}(\ell)\cap Q \cap P| \lesssim \d^{-5/2} \M_P(\d) = A\Delta |P|
    \]
    where we used that $T_{2\d}(\ell) \cap Q$ can be covered by $\sim \d^{-1/2}$ $\d$-balls. 
    So we can estimate 
    \begin{align*}
        \langle 1_\ell * (\eta_\d - \eta_{ \d/2 }), 1_{P} \rangle = \sum_{Q\in \mc Q:~ Q\cap T_{2\d}(\ell) \neq \emptyset } \langle 1_\ell * (\eta_\d - \eta_{ \d/2 }), 1_{P[Q]} \rangle \\\le \sum_{Q\not \in \mc Q(\ell)}\tau\Delta|P| +\sum_{Q\in \mc Q(\ell)} CA \Delta |P| \le C|P|(\tau + |\mc Q(\ell)| A\Delta)
    \end{align*}
    where we estimated the number of $Q\cap T_{2\d}(\ell)\neq \emptyset$ with $Q\not\in \mc Q(\ell)$ by $C\Delta^{-1}$.
    Summing this over all lines $\ell\in L$ gives for some constant $C$
    \begin{equation}\label{eq:l-incidence}
    B(\d) - B(\d/2) \le C \tau + C A |L|^{-1} \sum_{\ell\in L}\Delta |\mathcal Q_\ell|.    
    \end{equation}
    For the sake of contradiction, let us assume that $B(\d) -B(\d/2) \ge 2 C\tau$ holds. From (\ref{eq:l-incidence}) we then get
    \begin{equation}\label{eq:lower-bound-lambda}
    |L|^{-1} \sum_{\ell\in L}\Delta |\mathcal Q_\ell| \ge \tau A^{-1} =: \lambda_0.
    \end{equation}
    For $T \in \T$ let us define $\mathcal {Q}(T) = \bigcup_{\ell \in L[T]} \mathcal Q(\ell)$. 
    By pigeonholing, we can find some dyadic $\lambda \in [\lambda_0, C]$ and a subset $\T' \subset \T$ such that $\left|\mathcal Q(T) \right| \sim \lambda \Delta^{-1}$
    for $T \in \T'$ and $|\T'|\gtrsim \frac{\lambda_0}{\lambda \log(C/\lambda_0)} |\T|$.
    
    Decompose $\T' = \bigsqcup_{T_\Delta\in \T_\Delta} \T'[T_\Delta]$ where $\T_\Delta$ is a collection of essentially distinct $\Delta$-tubes covering $\T'$. Let $L[T_\Delta]$ be the union of $L[T]$ over all $T \in \T'[T_\Delta]$.
    
    Let $\mathbb U$ be a maximal by inclusion set of essentially distinct $2\d\times 2\d\times 4\Delta$-tubes such that for every $U \in \mathbb U$ there exists $T \in \T'$ and $Q\in \mathcal Q(T)$ so that $T \cap 2Q \subset U$. We estimate the size of $\mathbb U$ in two different ways. First, for $T_\Delta \in \T_\Delta$ define $\mathbb U[T_\Delta]$ to be the set of segments $U$ contained in $2T$. Observe that since tubes in $\T_\Delta$ are essentially distinct, the sets $\mathbb U[T_\Delta]$ are finitely-overlapping and so
    \begin{equation}\label{eq:U-lower-bound}
    |\mathbb U| \gtrsim \sum_{\T_\Delta} |\mathbb U[T_\Delta]|.
    \end{equation}
    Second, for $Q \in \mathcal Q$ set $\mathbb U[Q]$ be the set of $U$ contained in $2Q$, then we have an upper bound
    \begin{equation}\label{eq:U-upper-bound}
    |\mathbb U| \lesssim \sum_{\mc Q} |\mathbb U[Q]|.
    \end{equation}

    Fix $T_\Delta\in \T_\Delta$.
    Define a shading $Y=Y_{T_\Delta}$ on $\T'[T_\Delta]$ as follows: for a tube $T$ we set 
    \[
    Y(T) = \bigcup_{Q \in \mathcal Q(T)} T\cap 2Q
    \]
    By the construction of $\T'$, the shading $Y$ has density at least $\lambda$.

    Let $\psi=\psi_{T_\Delta}$ be an affine map which rescales $T_\Delta$ into the unit cube and define $\tilde \T$ to be the set of $\Delta=\d/\Delta$-tubes $\psi(T)$, $T\in \T'[T_\Delta]$. 
    Since $\T'\subset \T$ and $\T$ is $(t_1, t_2, CK)$-Katz--Tao by assumption, it follows that $\tilde \T$ is a $(t_1,t_2, CK)$-Katz--Tao set of $\Delta$-tubes. 

    % Let $K$ be the $t$-Katz-Tao constant of $\tilde \T$, i.e.
    % \[
    % K = \max_{\tilde U} \frac{|\tilde \T \cap \tilde U|}{ (|\tilde U| /|\tilde T|)^t }
    % \]
    % where maximum is over all convex sets $\tilde U \subset \R^3$, which we may restrict to the set of $\Delta/u\times \Delta/w\times 1$ boxes with $u, w \in (\Delta, 1)$. By definition, $\tilde \T$ is then a $(t, CK)$-Katz-Tao set of tubes.

    % We can estimate $K$ using the concentration constants of $L$ and the fact that each tube $T \in \T$ contains at least $\gtrsim \M_L(\d\times \d\times 1)/C_0$ lines from $L$. Indeed, let $\tilde U$ be a the $\Delta/u\times \Delta/w \times 1$ box inside the unit cube , then if $U$ is the $\d/u\times \d/w\times 1$ box preimage of $\tilde U$ then
    % \[
    % |\tilde \T \cap U| \lesssim |\T'[T_\Delta] \cap U| \lesssim C_0 \frac{|L[T_\Delta] \cap U|}{\M_L(\d\times \d\times 1)} \lesssim C_0 \frac{\M_L(\d/u\times \d/w\times 1)}{\M_L(\d\times \d\times 1)}
    % \]
    % Note that $|\tilde U|/|\tilde T| = (uw)^{-1}$, so we get
    % \begin{equation}\label{eq:K-upper-bound}
    % K \lesssim C_0 \max_{u, w\in (\Delta, 1)} (uw)^t \frac{\M_L(\d/u\times \d/w\times 1)}{\M_L(\d\times \d\times 1)}.
    % \end{equation}

    By Lemma \ref{lem:space-brush} applied to $\tilde \T$ and the rescaled shading $\tilde Y$ (which still has density at least $\lambda$) we obtain
    \[
    \left| \bigcup_{\tilde \T} \tilde Y(\tilde T)  \right| \gtrsim_\epsilon \d^{\epsilon} K^{-\frac{2+t_1}{2t_1+2t_2}} \lambda^{5/2} \Delta^2 |\tilde \T|^{\frac{2+t_1}{2t_1+2t_2}}.
    \]
    Since $\psi$ increases the volume by a factor of $\Delta^{-2}$ we have
    \[
    |\mathbb U[T_\Delta]| (\d \times \d \times \Delta) \gtrsim \left| \bigcup_{\T'[T_\Delta]} Y(T)  \right| \sim \Delta^2 \left| \bigcup_{\tilde \T} \tilde Y(\tilde T)  \right|
    \]
    and so we get a lower bound
    \begin{equation}\label{eq:U-lower-bound-2}
        |\mathbb U[T_\Delta]| \gtrsim_\epsilon \d^{\epsilon} K^{-\frac{2+t_1}{2t_1+2t_2}} \lambda^{5/2} \Delta^{-1} |\T'[T_\Delta]|^{\frac{2+t_1}{2t_1+2t_2}}
    \end{equation}

    We have for every $T_\Delta \in \T_\Delta$ that $|\T'[T_\Delta]| \lesssim \frac{\M_L(\Delta \times \Delta \times 1)}{\M_L(\d\times \d\times 1)} =: M$. So since $\alpha:=\frac{2+t_1}{2t_1+2t_2} \in [0,1]$, we obtain $\sum |\T'[T_\Delta]|^\alpha \gtrsim M^\alpha(|\T'|/M) = M^{\alpha-1} |\T'|$ and so
    \begin{equation}\label{eq:U-lower-bound-3}
    |\mathbb U| \gtrsim \sum_{\T_\Delta} |\mathbb U[T_\Delta]| \gtrsim K^{- \alpha} \lambda^{5/2} \d^{-1/2+2\epsilon} M^{\alpha-1} |\T'|.
    \end{equation}
    
    On the other hand, fix $Q \in \mathcal Q$ and let $\psi=\psi_Q: Q\rightarrow [0,1]^3$ be the affine rescaling map. Let $P^Q = \psi(P[Q])$. Let $L_Q$ to be the set of lines $\ell \in L'$ is such that $Q \in \mathcal Q(\ell)$ and let $L[Q] \subset L_Q$ be a maximal $\sim \d$-separated collection. Note that we have $|\mathbb U[Q]| \lesssim |L[Q]|$. 
    Define $L^Q$ to be the set of lines $\psi(\ell)$ with $\ell \in L[Q]$. Observe that by a change of variables for any $w >0$:
    \begin{align*}
    B(w; P^Q, L^Q) = \left\langle |P^Q|^{-1} 1_{P^Q} * \eta_w, |L^Q|^{-1} \sum_{L^Q} 1_\ell\right\rangle = \Delta^{-2}\left\langle |P[Q]|^{-1} 1_{P[Q]} * \eta_{w\Delta}, |L[Q]|^{-1} \sum_{L[Q]} 1_\ell \right\rangle 
    \end{align*}
    so using (\ref{eq:difference}) we get
    \begin{equation}\label{eq:BDelta-difference}
    B(\Delta; P^Q, L^Q)-B(\Delta/2; P^Q, L^Q) \ge \tau \frac{\Delta^3 |P|}{|P[Q]|}.    
    \end{equation}
    By Lemma \ref{lem:easy-high-low} we can upper bound this difference as 
    \begin{align*}
    |B(\Delta; P^Q, L^Q)-B(\Delta/2; P^Q, L^Q)|^2 \lesssim_{\epsilon'} \Delta^{-6-\epsilon'} \frac{\M_{P^Q}(\Delta)}{|P^Q|} \frac{\M_{L^Q}(\Delta \times \Delta \times 1)}{|L^Q|} \lesssim \\
    \lesssim \Delta^{-6-\epsilon'} \frac{\M_{P}(\d)}{|P[Q]|} |L[Q]|^{-1} \lesssim \Delta^{-\epsilon'} A \frac{|P|}{|P[Q]|} |\mathbb U[Q]|^{-1}. 
    \end{align*}
    since $L[Q]$ is a maximal $\d$-separated collection of lines. We obtain from (\ref{eq:BDelta-difference}) that
    \[
    \tau^2 \frac{\Delta^6 |P|^2}{|P[Q]|^2} \lesssim \Delta^{-\epsilon'} A \frac{|P|}{|P[Q]|} |\mathbb U[Q]|^{-1},
    \]
    \[
    |\mathbb U[Q]| \lesssim \tau^{-2} \Delta^{-6-\epsilon'} |P[Q]| A |P|^{-1}, 
    \]
    so summing over $Q$ and using (\ref{eq:U-upper-bound}) gives
    \[
    |\mathbb U|\lesssim \tau^{-2} \d^{-3-\epsilon'} A.
    \]
    So by (\ref{eq:U-lower-bound-3}) we have 
    \[
    \tau^{-2} \d^{-3-\epsilon'} A \gtrsim |\mathbb U| \gtrsim K^{-\alpha}\lambda^{5/2} \d^{-1/2+2\epsilon} M^{\alpha-1} |\T'|,
    \]
    \begin{equation}\label{eq:lambda-upper-bound}
    \lambda^{5/2} \lesssim K^{\alpha} \tau^{-2} \d^{-5/2-4\varepsilon} M^{1-\alpha} A |\T'|^{-1}    .
    \end{equation}
    % \blue{what's the power of $\lambda$?}
    % \begin{equation}\label{eq:lambda-upper-bound}
    % \lambda \lesssim K^{1/2} \tau^{-1} \d^{-5/4-2\epsilon} M^{\frac{1-\alpha}{2}} A^{1/2} |\T'|^{-1/2}    .
    % \end{equation}
    Using $\lambda \ge \lambda_0 = \tau A^{-1}$ and $|\T'| \gtrsim \d^{\epsilon} \frac{\lambda_0}{\lambda} |\T|$ we have
    \[
    \lambda^{5/2} |\T'| \gtrsim \d^{\varepsilon} \lambda_0^{5/2} |\T| = \d^{\varepsilon}  \tau^{5/2} A^{-5/2}
    \]
    and so (\ref{eq:lambda-upper-bound}) gives
    \[
    \d^{\varepsilon}  \tau^{5/2} A^{-5/2} |\T| \lesssim K^{\alpha} \tau^{-2} \d^{-5/2-4\varepsilon} M^{1-\alpha} A
    \]
    \[
    \tau^{9/2} \lesssim \d^{-5\varepsilon}K^{\alpha} M^{1-\alpha} A^{7/2} \d^{-5/2}|\T|^{-1}
    \]
    % \[
    % \d^{\epsilon/2} \tau A^{-1} |\T|^{1/2} \le \lambda |\T'|^{1/2} \lesssim K^{1/2} \tau^{-1} \d^{-5/4-2\epsilon} M^{\frac{1-\alpha}{2}} A^{1/2}
    % \]
    % \[
    % \tau^2 \le \d^{-3\epsilon-5/4} K^{1/2} M^{\frac{1-\alpha}{2}} A^{3/2} |\T|^{-1/2}.
    % \]
    Now recall that $M = \frac{\M_L(\Delta \times \Delta \times 1)}{\M_L(\d\times \d\times 1)}$ and $|\T| \gtrsim \frac{|L|}{C_0\M_L(\d\times \d\times 1)}$. Using this we get
    \[
    \tau^{9/2} \lesssim C_0\d^{-5\varepsilon}\d^{-5/2} K^{\alpha}A^{7/2}\frac{\M_L(\Delta \times \Delta \times 1)^{1-\alpha}\M_L(\d \times \d \times 1)^{\alpha}}{|L|}
    \]
    % \[
    % \tau^2 \lesssim_\epsilon K^{1/2} \d^{-\epsilon-5/4} A^{3/2} \M_L(\Delta \times \Delta \times 1)^{\frac{1-\alpha}{2}} \M_L(\d\times\d\times 1)^{\frac{\alpha}{2}}.
    % \]
    % % \[
    % \lesssim \d^{-\epsilon-5/4} A^{3/2} \M_L(\Delta \times \Delta \times 1)^{\frac{3t-2}{8t}} \M_L(\d\times\d\times 1)^{\frac{t+2}{8t}} \max_{u, w\in (\Delta,1)} \left(C_0 (uw)^t \M_L(\d/u \times \d/w \times 1) \right)^{\frac{2+t}{8t}}.
    % \]
    This matches the right hand side of (\ref{eq:second-high-low-improvement}), up to changing $\varepsilon$. So the difference $B(\d) - B(\d/2)$ is upper bounded by the desired quantity from above. Exactly the same argument shows that $B(\d)- B(\d/2)$ is lower bounded by the same quantity from below, giving us the desired upper bound on $|B(\d/2)-B(\d)|^2$.
\end{proof}

\section{Final preparations}\label{sec:preliminaries}

Recall from Section \ref{sec:intro} that $\Omega_d$ is the set of pairs $(p, \ell)$ where $p \in [0,1]^d$ and $\ell$ is a line passing through $p$. For a line $\ell$ in $\R^d$ we let $\theta(\ell) \in S^{d-1}$ denote a unit vector in the direction of $\ell$ (choosing arbitrarily between the two possibilities). 

%\subsection{Point-line configurations.} We will work with configurations of point-line pairs in two and three dimensions, so let us introduce some general notation for point-line pairs in $\R^d$. 
%We encode a pair $(p, \ell)$ such that $p \in \ell$ with a vector $(p, \theta) =(p_1, \ldots, p_d,\theta_1, \ldots, \theta_{d-1}) \in \Omega_d=[-1,1]^{2d-1}$, so that $\ell = \ell(p, \theta) := p + \R (\theta,1)$ is a line through $p$ with direction vector $(\theta,1)$. 
Given a configuration $X \subset \Omega_d$, we let $P[X] = \{p, ~(p, \ell)\in X\}$ and $L[X] = \{\ell, ~(p, \ell) \in X\}$ be the (multi-)sets of points and lines in $[0,1]^d$ determined by $X$. Furthermore, we let $\theta[X] = \{\theta(\ell),~(p, \ell) \in X\} \subset S^{d-1}$ be the multiset of directions of lines defined by $X$. 

%When we write `a point-line pair $p\in \ell$' we formally refer to the pair $(p, \theta)$ in the phase space $\Omega_d$ where $\theta=\theta(\ell)$ is so that $\ell = p+\R(\theta, 1)$. Note that not every line $\ell$ passing through $p$ corresponds to a direction vector $\theta \in [-1,1]^{d-1}$. However there is a finite collection of rigid motions $\rho_1, \ldots, \rho_{C}\in SO(d)$ such that for every $\ell \ni p$ we have $\theta(\rho_i(\ell)) \in [-1,1]^{d-1}$. So by rotating our coordinate system and passing to a subset of density at least $1/C$ we may always reduce to point-line pairs living in the phase space $\Omega_d$. 

% We define the {\em minimal distance }$d(X)$ of a point-line configuration $X$ as
% \[
% d(X) = \min_{(p, \theta)\neq (p', \theta')\in X} d(p, \ell(p', \theta')).
% \]

% \subsection{Using two-dimensional result}
% Let $\operatorname{PL}_d(\gamma)$ denote the following statement: for any point-line configuration $X$ in $[0,1]^d$ with $d(X) \ge \delta$, it follows that $|X| \lesssim_\varepsilon \delta^{-d+\gamma-\varepsilon}$ for every $\varepsilon>0$. So Theorem \ref{thm:main} is equivalent to $\operatorname{PL}_3(\gamma)$ for some $\gamma>0$. In the following proposition we show how two-dimensional information can be used to prove concentration estimates in three dimensions.

Recall the definition of the property $\operatorname{PL}_d(\gamma)$ from Section \ref{sec:intro}. In the next proposition we show that if we know $\operatorname{PL}_2(\gamma)$ for some $\gamma>0$ then for any point-line configuration $X \subset \Omega_3$ with $d(X)\ge \d$ the set of $\d$ tubes defined by $X$ satisfies a Katz--Tao axiom. 

\begin{prop}\label{prop:plane-reduction}
    Suppose that $\operatorname{PL}_2(\gamma)$ holds for some $\gamma \in [0,1]$. 

    If $X\subset \Omega_3$ is a point-line configuration such that $d(X) \ge \delta$, then for all $u\le w \in (0,1)$ with $uw \ge \d$ we have:
    \[
    \M_{L[X]}(u \times w\times 1) \lesssim_\varepsilon  \delta^{-3-\varepsilon} u^{1+\gamma} w^{2-\gamma}.
    \]
    In other words, the set of $\delta$-tubes determined by $L[X]$ is $(1+\gamma, 2-\gamma, C_\varepsilon\d^{-\varepsilon})$-Katz--Tao.
\end{prop}

\begin{proof}
    Let $\Pi \subset \R^3$ be an $u\times w\times 1$ prism and let $X_\Pi$ be the set of pairs $(p,\ell)\in X$ with $p\in \Pi$ and $|\ell \cap \Pi| \ge 1/2$. Let $P_\Pi = P[X_\Pi] \subset \Pi$ be the corresponding set of points. Denote $\alpha = c \delta/u$ for a small constant $c>0$.
    By pigeonhole principle, there exists a $\alpha u \times \alpha w \times \alpha$ box $\Pi'\subset \Pi$ homothetic to $\Pi$ such that 
    \[
    |P_{\Pi'}| =|P_\Pi \cap \Pi'| \gtrsim \alpha^3 |P_\Pi| \gtrsim (\delta/u)^3 |P_\Pi|.
    \]
    Let $X_{\Pi'}$ be the set of pairs $(p, \ell) \in X$ so that $p \in \Pi'$ and $|\ell\cap \Pi| \ge 1/2$. 
    Let $\psi: B' \to [0,1]^2$ be an affine map which collapses the side of length $\alpha u$ and rescales the other two by factors $(\alpha w)^{-1} \sim u/\delta w$ and $\alpha^{-1} \sim u/\delta$. 
    Observe that for $(p, \ell)\neq (p', \ell')\in X_{\Pi'}$ we have
    \begin{equation}\label{eq:psi-image}
    d(\psi(p), \psi(\ell')) \sim (u/\delta w) d(p, \ell') \gtrsim u/w.    
    \end{equation}
    So the point-line configuration $\tilde X = \psi(X_{\Pi'})$ in $[0,1]^2$ satisfies $d(\tilde X) \gtrsim u/w$. 

    We conclude using the property $\operatorname{PL}_2(\gamma)$ that 
    \[
    |X_{\Pi'}| \sim |\tilde X| \lesssim_\varepsilon (u/w)^{-2+\gamma-\varepsilon}.
    \]
    So we get
    \[
    \M_{L[X]}(u\times w\times 1) \le \max_{\Pi}|X_\Pi| \lesssim \alpha^{-3} |X_{\Pi'}| \lesssim  (\delta/u)^{-3} (u/w)^{-2+\gamma-\varepsilon} \lesssim_\varepsilon \delta^{-3-\varepsilon} u^{1+\gamma} w^{2-\gamma}.
    \]
\end{proof}

Proposition \ref{prop:plane-reduction} will be used to estimate the error terms in the high-low estimates from Section \ref{sec:high-low}. In order to prove $\operatorname{PL}_3(\gamma)$ for some $\gamma>0$ it is enough to assume that $\operatorname{PL}_2(\gamma_0)$ holds for some $\gamma_0>0$. So Proposition \ref{prop:main-in-R2} will suffice to prove Theorem \ref{thm:main} but if one were to optimize the value of $\gamma$ one could of course use the stronger statement $\operatorname{PL}_2(1/2)$ given by Theorem \ref{thm:point-line-R2}. 

\subsection{Uniformity}
It will be convenient to assume some nice regularity properties of $X$. The following setup and lemmas are closely related to the framework in \cite{cohen2025lower}. Some technical details are simplified here though (mainly due to the fact that we do not need to rescale into tubes). 
Let $X \subset \Omega_d$ be a point-line configuration in $\R^d$. 
For $u, v, w > 0$ let us define a concentration number 
\[
\M_X(u, v, w) = \max_{(p_0, \ell_0) \in \Omega_d} \# \{ (p, \ell) \in X:~ d(p, p_0) \le u, ~ d(\theta, \theta_0) \le v, ~d(\ell, \ell_0) \le w\},
\]
where $\theta = \theta(\ell)$ and $\theta_0 = \theta(\ell_0)$. For the purposes of this definition it is convenient to define a metric on the set of lines as follows:
\[
d(\ell, \ell') = d(\theta(\ell), \pm\theta(\ell')) + \min_{p\in \ell, p'\in \ell'} d(p, p'), 
\]
where we choose the sign so that this distance is minimized. In what follows we will omit this sign with the convention that $\theta(\ell)$ is always chosen appropriately. From this definition it is clear that we always have $d(\theta(\ell), \theta(\ell')) \le d(\ell, \ell')$ which implies the identity
\[
\M_X(u, v, w) = \M_X(u, \min(v, w), w).
\]

The quantity $\M_X$ generalizes the concentration constants of the point-set $P[X]$ and the line set $L[X]$, namely, we have the approximate formulas:
\begin{align*}
&\M_{P[X]}(u) \sim \M_X(u, 1, 1),\\
&\M_{L[X]}(w\times w\times 1) \sim \M_X(1, 1, w) = \M_X(1, w, w).
\end{align*}

Observe that the numbers $\M_X$ satisfy the following `Lipschitz' property for any $A, B, C\ge 1$:
\begin{equation}\label{eq:MX-lip}
    \M_X(Au, Bv, Cw) \lesssim A^{d} B^{d-1} C^{2d-2} \M_X(u, v, w).
\end{equation}
The proof is analogous to the proof of Proposition \ref{prop:easy-relations-M}.  
For parameters $K>0$ and $\Delta_1>\ldots>\Delta_m$, let us say that a configuration $X \subset \Omega_d$ is $K$-uniform on a sequence of scales $\Delta_1 > \ldots > \Delta_m$ if for all $i, j, k \in \{1, \ldots, m\}$ and for any $(p_0, \ell_0)\in X$ the number of $(p, \ell) \in X$ with $d(p, p_0) \le \Delta_i$, $d(\theta, \theta_0) \le \Delta_j$ and $d(\ell, \ell_0)\le \Delta_k$ is at least $\M_X(\Delta_i, \Delta_j, \Delta_k)/K$.

\begin{lemma}\label{lem:uniform-subset}
    For any $\d$ and $m \ge 1$ there exists some $K \le (C \log 1/\d)^{10m^3}$ such that the following holds. 
    For any $X \subset \Omega$ and any $\Delta_1> \ldots > \Delta_m \ge \d$ there exists $X' \subset X$ of size at least $K^{-1} |X|$ so that $X'$ is $K$-uniform on scales $\Delta_1, \ldots, \Delta_m$.
\end{lemma}

\begin{proof}
    This lemma is the same as Lemma 3.6 in \cite{cohen2025lower}, so let us only give a brief sketch. For each triple $(\Delta_i, \Delta_j, \Delta_k)$ be consider a metric on the space $\Omega_d$ given by 
    \[
    d_{i,j,k}((p, \theta), (p', \theta')) = \max \left\{\frac{d(p, p')}{\Delta_i}, \frac{d(\theta, \theta')}{\Delta_j}, \frac{d(\ell, \ell')}{\Delta_k} \right\}
    \]
    Note that $\M_X(\Delta_i, \Delta_j,\Delta_k)$ is essentially the maximum number of points of $X$ contained in a unit ball in metric $d_{i, j, k}$. 
    Now we fix a maximal $\sim 1$-net $\mathcal P_{i, j, k}\subset \Omega_d$ in the metric $d_{i, j, k}$ and define a hypergraph $H \subset X \times \prod \mc P_{i, j, k}$ as follows. For each $(p, \ell)\in X$ we put a tuple of the form $((p, \ell), (x_{i, j, k}) )$ in $H$ where $x_{i, j, k}$ is a closest (in metric $d_{i, j,  k}$) element of the net $\mc P_{i, j, k}$ to $(p, \ell)$. Now using a combinatorial lemma (see Lemma 3.5 in \cite{cohen2025lower}) we can pass to a large subgraph $H' \subset H$ which is approximately regular (the degrees of vertices in each component of $H'$ are the same up to a logarithmic factor). We then let $X' \subset X$ be the corresponding set of point-line pairs.
\end{proof}

For $\d>0, K \ge2$, we say that a point-line pair configuration $X$ is $(\d, K)$-uniform if the following two conditions hold:
\begin{itemize}
    \item[(i)] $X$ is $K$-uniform on the sequence of scales $K^{-1} > K^{-2} > \ldots >K^{-m}$ with $m = [\log_K(1/ \d)]$, 
    \item[(ii)] For each $j =1, \ldots, m$, the set of points $P[X]$ can be covered by $K^{-j}$-cubes which are pairwise $C_dK^{-j}$-separated for some fixed constant $C_d$ depending on the dimension. 
\end{itemize}

For any $X$ we can pass to a subset $X' \subset X$ of size $\gtrsim C_d^{-dm}|X|$ which has the property (ii) above. By applying Lemma \ref{lem:uniform-subset} to $X'$ we can then find a subset $X'' \subset X'$ which is $(\d, K)$-uniform and satisfies $|X''| \gtrsim (C_d\log( 1/\d))^{O(\varepsilon^{-3})}  |X|$ provided that $K \ge \d^{-\varepsilon}$. For a $(\d, K)$-uniform configuration $X$ we will typically restrict our attention to scales $w$ of the form $K^{-j}$ for some $j\ge 0$. Note however that due to (\ref{eq:MX-lip}), for any intermediate scales $w \in [K^{-j-1}, K^{-j}]$ we also have uniformity properties, albeit with an error $O(K^{O(1)})$ instead of $K$. 

Let $X$ be a $(\d,K)$-uniform configuration and fix an arbitrary pair $(p_0, \ell_0) \in X$. For $\Delta \in (\d,1)$ of the form $\Delta=K^{-j}$ we can define a rescaled configuration $X_\Delta$ as follows. Let $Q_\Delta$ be a $\Delta$-cube covering $p_0$ which is $C_dK^{-j}$-separated from other cubes covering $P[X]$ (coming from (ii) above). Let $\psi: Q_\Delta\rightarrow [0,1]^d$ be the homothetic rescaling map and define 
\[
X_\Delta = \psi(X \cap Q_\Delta) := \{(\psi(p), \psi(\ell)):~(p, \ell)\in X, ~p\in Q_\Delta\}.
\] 
Note that the definition of $X_\Delta$ depends on the choice of the initial pair $(p_0, \ell_0)$ but this choice will be unimportant due to uniformity of $X$. For two scales $\Delta, \Delta'$ we abuse notation and let $(X_\Delta)_{\Delta'} = X_{\Delta \Delta'}$, i.e. the rescalings of the point-line configuration $X_\Delta$ are compatible with those of $X$. 

\begin{claim}\label{claim:rescaling-stability}
    Suppose that $X$ is $(\d, K)$-uniform for $K\ge K_0(d)$ and let $\Delta = K^{-j} \in (\d, 1)$ then $X_\Delta$ is $(\d/\Delta, K^{C})$-uniform for some integer $C=C(d)\ge 1$. Furthermore, we have estimates
    \[
    \frac{1}{K^{C'}} \M_X(u\Delta, \min(v, w), w\Delta)\lesssim \M_{X_\Delta}(u, v, w) \le \M_X(u\Delta,\min( v,w), w\Delta)
    \]
    valid for $u, v,w$ of the form $K^{-j} \in [\d/\Delta, 1]$.
\end{claim}

\begin{proof}
    Clearly, we have $d(\psi(p), \psi(p')) = \Delta^{-1} d(p,p')$ for $p, p'\in \R^d$ and for lines $\ell, \ell'\subset \R^d$ the metric we introduced scales by $\psi$ as
    \[
    d(\psi(\ell), \psi(\ell')) = d(\theta(\ell), \theta(\ell')) + \min_{p\in \ell, p'\in \ell'} d(\psi(p), \psi(p')) = d(\theta(\ell), \theta(\ell')) + \Delta^{-1}\min_{p\in \ell, p'\in \ell'} d(p, p').
    \]
    Using these relations implies an upper bound
    \[
    \M_{X_\Delta}(u, v, w) \le \M_{X}(u\Delta, \min(v, w), w\Delta).
    \]
    On the other hand, if $(p_1, \ell_1) \in X\cap Q_\Delta$ then by the separation condition (ii), points $(p, \ell) \in X \setminus (X\cap Q_\Delta)$ will not contribute to the number of $(p_2, \ell_2)\in X$ so that $d(p_1, p_2) \le \Delta u$, $d(\theta_1, \theta_2) \le v$, $d(\ell_1, \ell_2) \le \Delta w$. Under these assumptions, we have $d(\psi(\ell_1), \psi(\ell_2)) \le v+w$ and so for $(\tilde p_1, \tilde \ell_1) = \psi(p_1,\ell_1)$ we get
    \begin{align*}
    \#\{ (\tilde p_2, \tilde \ell_2) \in X_\Delta:~ d(\tilde p_1,\tilde p_2) \le u, ~d(\tilde \theta_1, \tilde \theta_2) \le v, ~d(\tilde \ell_1, \tilde \ell_2) \le w \}\ge \\\ge \#\{ (p_2,\ell_2)\in X:~ d(p_1,p_2)\le  \Delta u,~ d(\theta_1, \theta_2) \le \min(v, \frac{w}{2}),~ d(\ell_1, \ell_2) \le \frac{\Delta w}{2} \}    
    \end{align*}
    So by $(\d, K)$ uniformity of $X$ for $K\ge 2$ this can be lower bounded by 
    \[
    \frac{1}{K} \M_X(\Delta u, \min(v, K^{-1}w), K^{-1} \Delta w) \gtrsim K^{-O(1)} \M_X(\Delta u, \min(v, w), \Delta w)
    \]
    where the latter follows by (\ref{eq:MX-lip}). So we conclude that $X_\Delta$ is $(\d/\Delta, K^{C})$-uniform for some (integer) constant $C$ depending on dimension (provided that $K$ is sufficiently large).  
\end{proof}

For $X \subset \Omega_d$, denote by $\theta[X] \subset S^{d-1}$ the multiset of directions determined by $X$. For a set $A$ in a metric space and $\d>0$ we let $|A|_\d$ denote the $\d$-covering number of $A$, i.e. the minimal number of $\d$-balls required to cover $A$. For uniform sets $X$, the covering numbers of related objects $P[X]$, $L[X]$, $\theta[X]$ can be determined using the concentration numbers $\M_X$:

\begin{prop}\label{prop:covering-numbers}
    Let $X \subset \Omega_d$ be a $(\d,K)$-uniform configuration and let $w = K^{-j}\in (\d, 1)$. Then we have
    \begin{align*}
        \frac{|X|}{\M_X(w, 1, 1)}\lesssim |P[X]|_w \lesssim K \frac{|X|}{\M_X(w, 1, 1)}\\
        \frac{|X|}{\M_X(1, 1, w)}\lesssim |L[X]|_w \lesssim K \frac{|X|}{\M_X(1, 1, w)}\\
        \frac{|X|}{\M_X(1, w, 1)}\lesssim |\theta[X]|_w \lesssim K \frac{|X|}{\M_X(1, w, 1)}
    \end{align*}
\end{prop}

\begin{proof}
    Let us prove the first line, the others are analogous. First, it is clear that $|P[X]|_w \gtrsim \frac{|X|}{\M_X(w, 1, 1)}$ since any $w$ ball contains at most $C \M_X(w, 1, 1)$ points from $P[X]$. 
    Let $\mc P \subset P[X]$ be a maximal by inclusion $w/2$-separated subset of points. Clearly, $w$-balls around $\mc P$ cover $P[X]$. Let $w' = K^{-j'}$ be a scale so that $w' \le cw$ for a small constant $c$. Then for each $p \in \mc P$ by uniformity we have 
    \[
    |B_d(p, w/2) \cap P[X]| \ge \frac{1}{K} \M_{X}(w', 1, 1) \gtrsim \frac{1}{K} \M_X(w, 1, 1).
    \]
    So it follows that $|\mc P| \lesssim K \frac{|X|}{\M_X(w, 1, 1)}$, concluding the argument.
\end{proof}

To lower bound the normalized incidence count $B(w) = \frac{I(w; P[X], L[X])}{w^{d-1} |X|^2}$ we can use the following `initial estimate'. Initial estimates like this were heavily used in arguments in \cite{cohen2023new} and \cite{cohen2025lower}.

\begin{prop}\label{prop:initial-estimate}
    Let $X \subset \Omega_d$ be a $(\d,K)$-uniform configuration and let $w = K^{-j}\in (C\d^{1/2}, 1)$. Then we have 
    \begin{equation}\label{eq:initial-estimate-general}
    B(w; P[X], L[X]) \gtrsim K^{-O(1)} \frac{|\theta[X_w]|_w}{w^{d-1}|L[X]|_w} \gtrsim K^{-O(1)} \frac{|\theta[X_w]|_w}{|\theta[X]|_w}.    
    \end{equation}
\end{prop}

\begin{proof}
    Recall that $I(w;P, L)$ is defined as $\sum_{p\in P, \ell\in L} \varphi(w^{-1} d(p, \ell))$ for some smooth symmetric bump function $\varphi$. We also know that this function satisfies $\varphi(t) \ge c$ for all $t\in [-c,c]$ for some constant $c$ depending on $d$. So we have a lower bound
    \begin{equation}\label{eq:incidence-lower-bound}
    I(w;P, L) \ge c\#\{ p\in P, \ell\in L:~ d(p, \ell) \le cw \}.    
    \end{equation}
    Let $w' =K^{-j'}$ be the largest scale such that $w'\le c'w$ for some constant $c'$. Consider the following set
    \[
    \Lambda = \{ ((p_0, \ell_0), (p_1, \ell_1), (p_2, \ell_2)) \in X\times X\times X:~ d(p_0, p_1) \le w', ~ d(\ell_1, \ell_2) \le w' \}.
    \]
    By uniformity, we have
    \[
    |\Lambda| \ge K^{-2} |X| \M_X(w', 1, 1) \M_X(1, 1, w') \gtrsim K^{-2} \frac{|X|^3}{|P[X]|_w |L[X]|_w}.
    \]
    On the other hand, for $((p_0, \ell_0), (p_1, \ell_1), (p_2, \ell_2))\in \Lambda$ we have $d(p_0, \ell_2) \le d(p_0, p_1) + d(p_1, \ell_2) \le d(p_0, p_1) + O(d(\ell_1, \ell_2)) \le O(w')$, where the implied constant depends on the metric we choose on $\mc A_{d,1}$. So if we choose the constant $c'$ so that $O(w') \le cw$, then $(p_0, \ell_2)$ gives a contribution to (\ref{eq:incidence-lower-bound}). Now if we fix $(p_0, \ell_0)$ and $(p_2, \ell_2)$ then the number of $(p_1,\ell_1)$ so that $((p_0, \ell_0), (p_1, \ell_1), (p_2, \ell_2))\in \Lambda$ can be upper bounded by
    \begin{align*}
    \#\{(p_1, \ell_1)\in X:~ d(p_0, p_1) \le w', d(\ell_1, \ell_2)\le w'\} \\\le \#\{(p_1, \ell_1)\in X:~ d(p_0, p_1) \le w', d(\ell_1, \tilde\ell_2)\le Cw'\}\\ \le \M_X(w',1, Cw') \le \M_X(w, 1, w)    
    \end{align*}
    where $\tilde \ell_2$ is a line passing through $p_0$, parallel to $\ell_2$. So by choosing $Cw'\le w$ we get
    \[
    I(w; P[X],  L[X]) \gtrsim \frac{|\Lambda|}{\M_X(w, 1, w)}.
    \]
    On the other hand, since $d(\ell, \ell') \le w$ implies $d(\theta(\ell), \theta(\ell'))\le w$ we have
    \[
    \M_X(w, 1, w)= \M_X(w, w, w) \lesssim K^{C'} \M_{X_w}(1, w, 1)
    \]
    where the last inequality follows from Claim \ref{claim:rescaling-stability}. Thus, using Proposition \ref{prop:covering-numbers}, we get
    \[
    I(w; P[X], L[X]) \gtrsim K^{-2-C'} \frac{|\Lambda|}{\M_{X_w}(1, w, 1)} \gtrsim K^{-5-C'}\frac{|X|^3 }{|P[X]|_w |L[X]|_w} \frac{|\theta[X_w]|_w}{|X_w|}
    \]
    we have $|P[X]|_w |X_w| \lesssim K |X|$, so after simplifying get:
    \[
    B(w) \gtrsim K^{-O(1)} \frac{|\theta[X_w]|_w}{w^{d-1}|L[X]|_w}.
    \]
    This gives the first inequality in (\ref{eq:initial-estimate-general}). The second inequality follows from the general fact that $|L|_w \lesssim |\theta(L)|_w w^{1-d}$. 
\end{proof}

\begin{prop}\label{prop:double-counting}
    Suppose that $X$ is $(\d, K)$-uniform and $w = K^{-j} \in (\d^{1/2},1)$. Then we have
    \[
    |L[X]|_w \gtrsim K^{-O(1)} w |\theta[X_w]|_w |P[X]|_w.
    \]
\end{prop}

\begin{proof}
    Since a $w$-tube can be covered by $\sim w^{-1}$ $w$-balls, we have a trivial upper bound
    \[
    I(w; P, L) \lesssim w^{-1}|L| \M_P(w) 
    \]
    and so $B(w) \lesssim w^{-d} \frac{\M_P(w)}{|P|}$. Comparing to (\ref{eq:initial-estimate-general}) we get
    \[
    K^{-O(1)} \frac{|\theta[X_w]|_w}{w^{d-1}|L[X]|_w}\lesssim w^{-d} \frac{\M_{P[X]}(w)}{|P[X]|} \lesssim K w^{-d} |P[X]|_w^{-1}
    \]
    and so by rearranging we get the desired bound.
\end{proof}

% \begin{prop}\label{prop:double-counting}
%     Let $X$ be a $(\d, K)$-uniform configuration and let $w=K^{-j} \in (\d,1)$. Then we have:
%     \[
%     |L[X]|_w \gtrsim \d^{C\epsilon} w |\theta[X_w]|_w |P[X]|_w 
%     \]
% \end{prop}

\section{Point-line configurations in $\R^3$}\label{sec:main-proof}

In this section we prove Theorem \ref{thm:main}. Let us restate the theorem in a form convenient for the proof.

\begin{theorem}\label{thm:main-restated}
    Suppose that $\operatorname{PL}_2(\gamma)$ holds for some $\gamma>0$. Then there exists some $\kappa(\gamma)>0$ such that $\operatorname{PL}_3(\kappa)$ holds. Explicitly, let $\d>0$ and
    let $X \subset \Omega_3$ be a point-line configuration with $d(X) \ge \d$. Then we have $|X| \lesssim_\varepsilon \d^{-\varepsilon}\d^{-3+\kappa}$ for every $\varepsilon>0$. 
\end{theorem}

We will prove Theorem \ref{thm:main-restated} using the following three lemmas, each giving a sufficient condition on $X$ under which we can improve on the trivial upper bound $|X| \lesssim \d^{-3}$. Recall that in Section \ref{sec:preliminaries} for a $(\d,K)$-uniform $X$ we introduced the rescaled point-line configurations $X_w$ for every $w = K^{-j} \in (\d, 1)$. Note that by definition we have
\[
d(X_w) \ge w^{-1} d(X)
\]
and that $|X_w| \gtrsim K^{-1} \M_X(w, 1, 1) \gtrsim w^3 |X|$. 

In the first lemma, we assume that for some fairly coarse scale $w$, the set of directions of $X$ is far from uniform on scale $w$. This condition is designed to make Corollary \ref{cor:refined-high-low} applicable on all scales and so it can be used to obtain a small improvement over the simple high-low. 

\begin{lemma}\label{lem:few-directions-coarse-scale}
    Suppose that $\operatorname{PL}_2(\gamma)$ holds for some $\gamma \in (0,1]$. Then the following holds with $A=A(\gamma) = \frac{100}{\gamma}$.

    Let $X \subset \Omega_3$ be a $(\d, K)$-uniform point-line configuration such that $d(X) \ge \d$. 
    Suppose that for some $w=K^{-i} \in (\d,1)$ and $\beta \in (0,1)$ we have $|\theta[X]|_w \le w^{-2+\beta}$ and $w \ge \d^{\frac{1}{A\log(2/\beta)}}$. Then $\d^3|X| \lesssim_\varepsilon K^{O(1)}\d^{-\varepsilon} w^{\beta/A}$. Here $\log$ denotes the natural logarithm.
\end{lemma}

\begin{proof}
    For every $j =0, \ldots, m:=[\log_w\d]$ we may consider the rescaled configuration $X_{w^j}$. 
    Denote by $\beta_j$ the number such that 
    \[
    w^{\beta_j} = w^2 |\theta[X_{w^j}]|_w.
    \]
    Note that the sequence $\beta_j$ is increasing: indeed, note that $\theta[X_{w^{j+1}}] \subset \theta[X_{w^{j}}]$ and so we have
    $\beta = \beta_0 \le \beta_1 \le \ldots \le \beta_m\le 2$. 
    
    Let $\tau > 1$ be a constant to be determined later. Suppose that we have $\tau^{m-2} \beta > 2$. Then it follows that there is $j \in \{0,\ldots, m-2\}$ such that 
    \[
    \beta_j \ge \tau^j\beta ,~~ \beta_{j+1} \le \tau\beta_j.
    \]
    Let $P_j = P[X_{w^j}]$ and $L_j = L[X_{w^j}]$. By Proposition \ref{prop:initial-estimate} we then get 
    \[
    B_j(w) = B(w; P_j, L_j) \gtrsim K^{-O(1)} \frac{|\theta[X_{w^{j+1}}]|_w}{|\theta[X_{w^{j}}]|_w} = K^{-O(1)} w^{\beta_{j+1}-\beta_j} \gtrsim K^{-O(1)} w^{(\tau-1)\beta_j}.
    \]
    Now we apply Corollary \ref{cor:refined-high-low} to the pair $(P_j, L_j)$ and on scale $v \in [c\d/w^j, w]$. Write $\d_j = \d/w^j$ for convenience (recall that $d(X_{w^j}) \ge \d_j$ from rescaling). Since $P_j$ is $\d_j$-separated, we have 
    \[
    \M_{P_j}(v) \lesssim (v/\d_j)^3.
    \]
    By Proposition \ref{prop:plane-reduction}, property $\operatorname{PL}_2(\gamma)$ implies for $u \in (v,1)$:
    \[
    \M_{L_j}(v \times v/u\times 1) \lesssim_\varepsilon \d^{-\varepsilon} \d_j^{-3} v^3 u^{-2+\gamma},
    \]
    giving $\M_{L_j}(v \times v/u\times 1) \le u^{-2+\gamma} M$ for $M \lesssim \d^{-\varepsilon} \d_j^{-3} v^3$.
    Lastly, we have
    \[
    |\theta(L_j)|_v \lesssim (w/v)^2 |\theta(L_j)|_w \lesssim w^{\beta_j} v^{-2},
    \]
    giving $ |\theta(L_j)|_v \le \nu v^{-2} $ with $\nu \sim w^{\beta_j}$.
    So by Corollary \ref{cor:refined-high-low} we get
    \begin{align*}
        |B(v) - B(2v)|^2 \lesssim_\varepsilon \d^{-\varepsilon} \nu^{\gamma/4} v^{-6} \frac{\M_{P_j}(v)}{|P_j|} \frac{M}{|L_j|} \lesssim \d^{-2\varepsilon} w^{\gamma \beta_j/4} (\d_j^{3} |X_j|)^{-2}. 
    \end{align*}

    So since $d(X_j) \ge \d_j$ we have $B_j(c \d_j) \lesssim \d_j^{-2} |X_j|^{-1}$ and using the initial estimate $B_j(w )\gtrsim K^{-O(1)} w^{(\tau-1)\alpha_j}$ we then obtain by summing the high-low errors over all $v$:
    \[
     K^{-O(1)} w^{(\tau-1)\beta_j} \lesssim \d_j^{-2} |X_j|^{-1} + \d^{-o(1)} w^{\gamma \beta_j / 8}(\d_j^{3} |X_j|)^{-1}.
    \]
    If we choose $\tau = 1+\gamma /16$ then this bound implies
    \[
    \d_j^3|X_j| \lesssim K^{O(1)}\max(\d^{-o(1)} w^{\gamma \beta_j/16}, \d_j w^{-\gamma \beta_j/16} ) \le  K^{O(1)}\d^{-o(1)} w^{\gamma \beta_j/16} \le K^{O(1)}\d^{-o(1)} w^{\gamma \beta/16}.
    \]
    We have $\d^3|X| \lesssim \d_j^3|X_j|$ so this gives the desired bound. For the argument to work, we need the condition $(1+\gamma/16)^{m-2} > 2/\beta$ where $m = [\log_w \d]$ which is satisfied if, say, $m \ge 100 \gamma^{-1} \log(2/\beta)$.
\end{proof}

In the second lemma we consider an opposite extreme: we assume that the restriction of $X$ onto every $\d^{1/2}$ box $Q$ spans almost all directions on scale $\d^{1/2}$. This condition allows us to apply Theorem \ref{thm:small-box-high-low} on all intermediate scales. %\blue{For $w$ not of the form $K^{-i}$ we need to apply (\ref{eq:MX-lip}) together with uniformity -- TODO}

\begin{lemma}\label{lem:many-directions-inside-box}
    Suppose that $\operatorname{PL}_2(\gamma)$ holds for some $\gamma \in (0,1]$. Then the following holds with $\beta_0=\beta_0(\gamma) = \gamma/200$.

    %For any $\epsilon_0 > 0$ there is $\epsilon >0$ and $\d_0$ such that the following holds for all $\d < \d_0$ and $\alpha, \gamma \in (0,1)$ with $106\alpha+96\gamma \le 1-\epsilon_0$.

    Let $X \subset \Omega_3$ be a $(\d, K)$-uniform configuration with $d(X)\ge \d$. Suppose that for some $\beta \in [0,\beta_0]$ we have $|\theta[X_{\Delta}]|_{\Delta} \ge \Delta^{-2} \d^{\beta}$ for some $\Delta =K^{-i}\in [C\d^{1/2}, CK\d^{1/2}]$. Then we have $\d^3|X| \lesssim K^{O(1)}\d^{\beta_0}$.
\end{lemma}

\begin{proof}
    By Proposition \ref{prop:initial-estimate} we have 
    \[
    B(\Delta) = B(\Delta; P[X], L[X]) \gtrsim K^{-O(1)} \frac{|\theta[X_{\Delta}]|_{\Delta}}{|\theta[X]|_{\Delta}} \gtrsim K^{-O(1)} \frac{\Delta^{-2}\d^{\beta}}{\Delta^{-2}} \gtrsim K^{-O(1)} \d^\beta.
    \] 
    Now let $w \in [c\d, \Delta]$ and let us estimate the difference $B(w/2) - B(w)$ using Theorem \ref{thm:small-box-high-low}. Let $P = P[X]$ and $L=L[X]$. It follows from uniformity and (\ref{eq:MX-lip}) that $L$ is $C_0\sim K^{O(1)}$ uniform on scale $w$. Let $\T_w$ be the set of $w$-tubes determined by $L$. We have by Proposition \ref{prop:covering-numbers} that $|\T_w| = K^{O(1)} \frac{|L|}{\M_L(w\times w\times 1)}$.
    
    By Proposition \ref{prop:plane-reduction}, property $\operatorname{PL}_2(\gamma)$ implies that 
    \[
    \M_{L}(u\times v\times 1) \lesssim_\varepsilon \d^{-\varepsilon} u^{1+\gamma} v^{2-\gamma} \d^{-3}.
    \]
    Therefore, using uniformity of $X$, we get that for any $u\times v\times 1$ box $\Pi$ we have
    \[
    |\T_w\cap \Pi| \lesssim K^{O(1)} \frac{\M_{L}(u\times v\times 1)}{\M_L(w\times w\times 1)} \lesssim\left( \d^{-\varepsilon} K^{O(1)} (w/\d)^{3} \M_L(w\times w\times 1)^{-1} \right) (u/w)^{1+\gamma} (v/w)^{2-\gamma}.
    \]
    We conclude that $\T_w$ is a $(1+\gamma, 2-\gamma, K_w)$-Katz--Tao set of $w$-tubes where $K_w \sim \d^{-\varepsilon} K^{O(1)} (w/\d)^{3} \M_L(w\times w\times 1)^{-1}$.

    Note that we have $|P|_\Delta \gtrsim (\d/\Delta)^3 |P|_\d \gtrsim \Delta^{-3} (\d^3|X|)$.
    By Proposition \ref{prop:double-counting} we have
    \[
    |L|_{\Delta} \gtrsim K^{-O(1)} \Delta |\theta[X_{\Delta}]|_{\Delta} |P|_{\Delta} \gtrsim K^{-O(1)} \Delta^{-4} \d^{\beta}(\d^3|X|)
    \]
    and so by Proposition \ref{prop:easy-relations-M} we obtain  
    \begin{equation}\label{eq:Llargecovering}
    |L|_{w^{1/2}} \gtrsim K^{O(1)}(\Delta/w^{1/2})^4 |L|_\Delta \gtrsim K^{-O(1)}w^{-2} \d^{\beta} (\d^3|X|).    
    \end{equation}
    Using $(\d, K)$-uniformity of $X$ this implies
    \[
    \M_L(w^{1/2}\times w^{1/2} \times 1) \lesssim K^{O(1)}w^2 \d^{-\beta} |L| (\d^3|X|)^{-1}.
    \]
    So by Theorem \ref{thm:small-box-high-low} applied with $t_1=1+\gamma$, $t_2=2-\gamma$ we have $\alpha = \frac{t_1+2}{2t_1+2t_2} = \frac{3+\gamma}{6}$ and
    \begin{align*}
        |B(w/2) &- B(w)|^{9/2}\lesssim \\&\lesssim_{\varepsilon} K^{O(1)} \d^{-\varepsilon} w^{-5/2} K_w^\alpha (\d^3|X|)^{-7/2} \frac{\M_L(w^{1/2} \times w^{1/2} \times 1)^{1-\alpha}\M_L(w \times w \times 1)^{\alpha}}{|L|} \\
        &\lesssim K^{O(1)}\d^{-2\varepsilon}w^{-5/2} (w/\d)^{3\alpha} (\d^3|X|)^{-7/2}  \frac{\M_L(w^{1/2} \times w^{1/2} \times 1)^{1-\alpha}}{|L|} \\
        &\lesssim K^{O(1)}\d^{-2\varepsilon}w^{-5/2} (w/\d)^{3\alpha}(\d^3|X|)^{-9/2+\alpha} (w^{2} \d^{-\beta})^{1-\alpha}|L|^{-\alpha}\\
        &\lesssim K^{O(1)}\d^{-2\varepsilon} \d^{-\beta(1-\alpha)}(\d^3|X|)^{-9/2} w^{\alpha-\frac{1}{2}}
    \end{align*}
    Now recall that $w\le \d^{1/2}$ and $\alpha -\frac12 = \frac\gamma 6$:
    \begin{equation}\label{eq:Bw-difference-Lem2}
    |B(w/2) - B(w)| \lesssim K^{O(1)}\d^{-\varepsilon} \d^{-\beta/9} (\d^3|X|)^{-1} \d^{\gamma/54}.    
    \end{equation}
    Recalling $B(c\d) \lesssim \d^{-2}|X|^{-1}$ (which follows from $d(X)\ge \d$) we conclude that
    \[
    B(\Delta) \lesssim \max_{w}(B(c\d), |B(w)-B(w/2)|)
    \]
    \[
    \d^\beta \lesssim K^{O(1)}\d^{-\varepsilon}\max( (\d^3|X|)^{-1} \d^{\gamma/54 - \beta/9},  \d^{-2}|X|^{-1})
    \]
    \[
     \d^3|X| \lesssim  K^{O(1)}\d^{-\varepsilon}\max\{\d^{\gamma/54 - 10\beta/9}, \d^{1-\beta}  \} \lesssim \d^{\beta_0}
    \]
    if we define $\beta_0=\gamma/200$. This completes the proof.
\end{proof}

Lastly, in the third lemma we consider an intermediate situation where $X$ does not span all directions on some scale $\rho$ but restrictions of $X$ onto $\rho^{1/2}$-boxes span almost all directions on scale $\rho^{1/2}$. This condition interpolates between the two conditions in previous lemmas and it allows us to combine Corollary \ref{cor:refined-high-low} and Theorem \ref{thm:small-box-high-low} to cover all scales.

\begin{lemma}\label{lem:intermediate-case}
    Suppose that $\operatorname{PL}_2(\gamma)$ holds for some $\gamma \in (0,1]$. Then the following holds with $A=A(\gamma)=\frac{100}{\gamma}$.
    %There is a constant $\beta_0=\beta_0(\gamma)$ such that the following holds.

    Let $X\subset \Omega_3$ be a $(\d, K)$-uniform  configuration such that $d(X) \ge \d$. Let $w =K^{-2i}\in (\d,1)$ and suppose that $|\theta[X]|_w \le w^{-2}\d^{\sigma}$ and $|\theta[X_{w^{1/2}}]|_{w^{1/2}} \ge w^{-1}\d^{\beta}$ holds for some $\beta, \sigma \in [0,1]$. 
    
    %Then for $\gamma>0$ such that $\rho < \d^{106\alpha+96\gamma+\epsilon_0}$ and $\beta\ge 8(\alpha+\gamma)+\epsilon_0$ we have $|X| \lesssim \d^{-3+\gamma}$.
    Suppose that $w \le \d^{2A\beta}$ and $\sigma \ge A \beta$ and $\beta\le 1/2$.
    Then $\d^3|X| \lesssim K^{O(1)} \d^{-\varepsilon}\max(w^{1/A}, \d^{ \sigma/A})$.
\end{lemma}

\begin{proof}
    The proof essentially follows by combining the computations we did in Lemma \ref{lem:few-directions-coarse-scale} and Lemma \ref{lem:many-directions-inside-box}. First, by Proposition \ref{prop:initial-estimate} we have $B(Cw^{1/2}) \gtrsim K^{-O(1)} \d^\beta$ and 
    by repeating the proof of Lemma \ref{lem:many-directions-inside-box} with $w$ and $w^{1/2}$ in place of $\d$ and $\Delta$, we can show that for $v \in [cw, Cw^{1/2}]$ we have the upper bound (cf. (\ref{eq:Bw-difference-Lem2}))
    \[
    |B(v/2) - B(v)| \lesssim K^{O(1)} \d^{-\varepsilon} \d^{-\beta/9} (\d^3 |X|)^{-1} w^{\gamma/54}
    \]

    For $v \in [c\d, \rho]$ we have $|\theta[X]|_v \lesssim (w/v)^2 |\theta[X]|_w \lesssim v^{-2} \d^{\sigma}$. So using Corollary \ref{cor:refined-high-low} we have 
    \[
    |B(v) - B(2v)|^2 \lesssim \d^{-\varepsilon} \d^{\sigma \gamma/4} v^{-6} \frac{\M_{P[X]}(v)}{|P|} \frac{M}{|L|} \lesssim \d^{-\varepsilon} \d^{\sigma \gamma/4} (\d^3 |X|)^{-2}.
    \]
    So by summing over $v \in (c\d, C\rho^{1/2})$ we obtain
    \[
    |B(c\d) - B(Cw^{1/2})| \lesssim K^{O(1)} \d^{-\varepsilon} (\d^3 |X|)^{-1} \max( \d^{-\beta/9}  w^{\gamma/54},  \d^{\gamma\sigma /8})
    \]
    and thus by using  $B(c\d) \lesssim \d^{-2}|X|^{-1}$ we conclude
    \[
    \d^3|X| \lesssim K^{O(1)} \d^{-\varepsilon}  \max(\d^{-10\beta/9} w^{\gamma/54},~\d^{-\beta} \d^{\gamma\sigma/8} ,~ \d^{1-\beta}  )
    \]
    Using the restrictions on $\beta$, $\sigma$, $\rho$ we get the desired bound on $\d^3|X|$.
    \end{proof}

Now we are ready prove Theorem \ref{thm:main-restated} (and thus Theorem \ref{thm:main} as well).

 \begin{proof}[Proof of Theorem \ref{thm:main-restated}]
     
Suppose that $\operatorname{PL}_2(\gamma)$ holds for some $\gamma>0$ and fix a point-line configuration $X$ with $d(X)\ge \d$. Let $A= \frac{100}{\gamma}$ and $\beta_0 = \gamma/200$. 

% and suppose that $|X| \ge \d^{-3+\kappa}$ for a tiny constant $\kappa$. 

Let $\varepsilon>0$ be an arbitrarily small constant and denote $K=\d^{\varepsilon}$. By Lemma \ref{lem:uniform-subset} and the comments after it we can pass to a subset $X' \subset X$ which is $(\d, K)$-uniform and $|X'| \gtrsim (\log1/\d)^{O(\varepsilon^{-O(1)})} |X|$. So without loss of generality, we may replace $X$ with $X'$ and assume that $X$ is $(\d, K)$-uniform. Let us say that a scale $w\in [\d, 1]$ is admissible if $w=K^{-i}$ for an integer $i$.
Then for every admissible $w \in (\d, 1)$ we have a rescaled point-line configuration $X_w$ which satisfies $d(X_w) \ge \d/w$. From definition of $X_w$, we have 
\begin{equation}
    \d^3 |X| \lesssim K^{O(1)} (\d/w)^3|X_w|
\end{equation}
so it suffices to upper bound the size of $|X_w|$ for some $w$. Suppose that $\d^3|X| \ge \d^{\kappa}$ for some constant $\kappa>0$, our goal is to estimate this constant. Note that by Claim \ref{claim:rescaling-stability} the configuration $X_w$ is $(\d/w, O(K^{O(1)}))$-uniform and so the lemmas from section can be applied to $X_w$. When we apply lemmas to $X_w$, we define the rescaled configurations by $(X_w)_u = X_{wu}$.

Let $\rho_0 =\d^{r}$ for some $r\in (0,1/4)$ be a large admissible scale and let $\beta^* \in (0,1)$ be a small constant to be determined later (see the end of the proof where we choose these two parameters). 
Let $\lambda = \log(2/\beta^*)$. Let $w_0$ be the smallest admissible scale such that $w_0 \ge \rho_0^{\frac{1}{A\lambda}}$ (so we have the upper bound $w_0 \le K\rho_0^{\frac{1}{A\lambda}}$).
Then by Lemma \ref{lem:few-directions-coarse-scale} applied to $X_{\d/\rho_0}$ we get that at least one of the following two options holds:
\begin{itemize}
    \item[(i)] $\rho_0^3 |X_{\d/\rho_0}| \lesssim K^{O(1)} \rho_0^{-\varepsilon} w_0^{\beta^*/A}$ or
    \item[(ii)] $|\theta[X_{\d/\rho_0}]|_{w_0} \ge w_0^{-2+\beta^*}$.
\end{itemize}
In the first case it follows that
\[
    \d^{\kappa} \lesssim K^{O(1)} \rho_0^{-\varepsilon} w_0^{\beta^*/A} \lesssim K^{O(1)} \d^{-\varepsilon} \rho_0^{\frac{\beta^*}{A^2 \log(2/\beta^*)}},
\]
which then gives
\begin{equation}\label{eq:first-exit}
    \kappa \ge \frac{r\beta^*}{A^2 \lambda} + O(\varepsilon)
\end{equation}
As long as $\beta^*$ and $r$ are taken to be constants depending only on $\gamma$ this will give a sufficient lower bound on $\kappa$. 

Let us now consider the second case. Let $\rho_1 = \rho_0w_0$ and $w_1 = w_0^2$. Note that using our convention we have $(X_{\d/\rho_1})_{w_0} = X_{\d/\rho_0}$.
Define $\beta_1$ to be the number such that $\rho_1^{\beta_1} = w_0^{\beta^*}$. So by rearranging and using that $\log_{\rho_0} w_0 = \frac{1}{A\lambda} + O(\varepsilon)$, we get 
\[
\beta_1 = \frac{\beta^*}{A\lambda+1} + O(\varepsilon).
\]
In particular it is clear that $\beta_1\le 1/2$. 
So we get that 
\[
|\theta[(X_{\d/\rho_1})_{w_1^{1/2}}]|_{w_1^{1/2}} \ge w_1^{-1} \rho_1^{\beta_1}. 
\]
Let $\sigma_1$ be the number such that
\[
|\theta[X_{\rho_1}]|_{w_1} = w_1^{-2} \rho_1^{\sigma_1}.
\]
So by Lemma \ref{lem:intermediate-case} applied to $X_{\d/\rho_1}$ and $w=w_1$: if $w_1 \le \rho_1^{2A\beta_1}$ and $\sigma_1 \ge A\beta_1$ then we have 
\[
\rho_1^3 |X_{\d/\rho_1}| \lesssim K^{O(1)} \rho_1^{-\varepsilon} \max(w_1^{1/A}, \rho_1^{\sigma_1/A}) \le K^{O(1)} \d^{-\varepsilon} \max(w_0^{2/A}, w_0^{\beta^*})
\]
where we used $\sigma_1\ge A\beta_1$ and $\rho_1^{\beta_1} = w_0^{\beta^*}$, so this leads to
\begin{equation}\label{eq:second-exit}
    \kappa \ge \frac{r}{A\lambda} \max(2/A, \beta^*) +O(\varepsilon).
\end{equation}
Note that we have $\rho_1^{2A\beta_1} = w_0^{2A\beta^*}$ so if $\beta^* \le 1/A$ then this implies $w_1 \le \rho_1^{2A\beta_1}$. So if Lemma \ref{lem:intermediate-case} does not apply, then we must have $\sigma_1 \le A\beta_1$, i.e.
\begin{equation}\label{eq:start-iteration}
|\theta[X_{\d/\rho_1}]|_{w_1} \ge w_1^{-2} \rho_1^{A\beta_1} = w_1^{-2} w_0^{A\beta^*}.    
\end{equation}
This condition is analogous to the one in alternative (ii) above. We can now iterate this argument as follows. Define a pair of scales $(\rho_j, w_j)$ by the rule 
\[
\rho_j = \rho_{j-1} w_{j-1}, \quad w_j = w_{j-1}^2, 
\]
which gives $w_j = w_0^{2^j}$ and $\rho_j = \rho_0 w_0^{2^{j}-1}$, $j\ge 0$. Suppose that we have for some $j\ge 1$ that
\begin{equation}\label{eq:iteration-condition}
    |\theta[X_{\d/\rho_j}]|_{w_j} \ge w_j^{-2} w_0^{A^{j} \beta^*}
\end{equation}
and let us try to prove (\ref{eq:iteration-condition}) for $j+1$. Suppose that $j$ is such that $\rho_{j+1} \ge \d^{1/2}$.
Let $\beta_j$ be the number so that $\rho_{j+1}^{\beta_{j+1}} = w_0^{A^{j} \beta^*}$. This then implies by (\ref{eq:iteration-condition}) that
\[
|\theta[(X_{\d/\rho_{j+1}})_{w_{j+1}^{1/2}}]|_{w_{j+1}^{1/2}} \ge w_{j+1}^{-1} \rho_{j+1}^{\beta_{j+1}}.
\]
Let $\sigma_{j+1}$ be the number so that 
\[
|\theta[X_{\d/\rho_{j+1}}]|_{w_{j+1}} = w_{j+1}^{-2} \rho_{j+1}^{\sigma_{j+1}}
\]
Then by Lemma \ref{lem:intermediate-case}, if we have $\beta_{j+1}\le 1/2$, $w_{j+1} \le \rho_{j+1}^{2A \beta_{j+1}}$ and $\sigma_{j+1} \ge A\beta_{j+1}$ then it follows that
\[
\rho_{j+1}^3 |X_{\d/\rho_{j+1}}| \lesssim K^{O(1)} \rho_{j+1}^{-\varepsilon} \max (w_{j+1}^{1/A}, \rho_{j+1}^{\sigma_{j+1}/A}) \le K^{O(1)} \rho_{j+1}^{-\varepsilon}\max (w_{0}^{2^{j+1}/A}, w_0^{A^j \beta^*})
\]
where we used the lower bound on $\sigma_{j+1}$ and the definition of $\beta_{j+1}$. So this gives 
\begin{equation}\label{eq:step-j-exit}
    \kappa \ge \frac{r}{A\lambda} \max(2^{j+1}/A, A^j \beta^*)+O(\varepsilon).
\end{equation}
So we may assume that conditions of Lemma \ref{lem:intermediate-case} do not apply. If $\sigma_{j+1} \ge A\beta_{j+1}$ then this precisely means that (\ref{eq:iteration-condition}) holds with $j$ replaced by $j+1$ and so we completed an iteration step. Otherwise, it could happen that one of $\beta_{j+1} > 1/2$, or $w_{j+1} > \rho_{j+1}^{2A\beta_{j+1}}$ or $\rho_{j+1}<\d^{1/2}$ holds. Let us check what restrictions this puts on $j$. By definition,
\[
\rho_{j+1} = \rho_0 w_0^{2^{j+1}-1} = \rho_0^{ 1 + (2^{j+1}-1) /A\lambda + O(\varepsilon) } = \d^{r(1 + (2^{j+1}-1) /A\lambda)+O(\varepsilon)}
\]
so we have $\rho_{j+1}\ge \d^{1/2}$ as long as $r(1 + 2^{j+1} /A\lambda) \le 1/3$ and $\varepsilon$ is much smaller than all other parameters. Next, we have 
\[
\rho_{j+1}^{\beta_{j+1}} = w_0^{A^j \beta^*}
\]
\[
\beta_{j+1} = \frac{A^{j} \beta^* }{A\lambda +2^{j}} + O(\varepsilon)
\]
so we will have $\beta_{j+1} \le 1/2$ as long as, say, $A^j\beta^* \le 1/2$ holds. Finally, the condition $w_{j+1} \le \rho_{j+1}^{2A\beta_{j+1}}$ is equivalent to 
\[
w_0^{2^{j+1}} \le w_0^{2 A^{j+1}\beta^*}
\]
i.e. we need $A^{j+1}\beta^* \le 2^j$. We conclude these observations in the following proposition:
\begin{prop}\label{prop:conditions-for-iteration}
    Suppose that $j\ge 1$ is such that $A^{j+1} \beta^* \le 1$ and $2^{j}r \le 1$. Then (\ref{eq:iteration-condition}) implies that either (\ref{eq:iteration-condition}) holds with $j+1$ instead of $j$ or (\ref{eq:step-j-exit}) holds.
\end{prop}

Now suppose that (\ref{eq:iteration-condition}) holds for some $j$. Note that we have $\rho_j = \rho_0 w_0^{2^j-1} \le w_0^{2^j}= w_j$ and $\rho_j/w_j = \rho_0/w_0 \ge \rho_0$.
Observe that $(X_{\d/\rho_{j}^2})_{\rho_j} = X_{\d/\rho_j}$ and so we have 
\[
|\theta[(X_{\d/\rho^2_{j}})_{\rho_j}]|_{\rho_j} = |\theta[X_{\d/\rho_{j}}]|_{\rho_j} \ge |\theta[X_{\d/\rho_{j}}]|_{w_j} \ge w_j^{-2} w_0^{A^j\beta^*} \ge  \rho_j^{-2} \rho_0^2 w_0^{A^j\beta^*}
\]
where we used a trivial relation between covering numbers. 
Recall that we defined $\beta_0 = \beta_0(\gamma) = \gamma/200$. Suppose that the following inequality holds 
\begin{equation}\label{eq:last-step}
    \rho_0^2 w_0^{A^j\beta^*} \ge \rho_j^{2\beta_0}.
\end{equation}
So by Lemma \ref{lem:many-directions-inside-box} applied to $X_{\d/\rho_{j}^2}$ (with $\Delta = \rho_j$) it then follows that $(\rho_j^2)^3 |X_{\d/\rho_j^2}| \lesssim K^{O(1)} \rho_{j}^{2\beta_0}$ which leads to 
\[
\d^\kappa \lesssim K^{O(1)} \rho_{j}^{2\beta_0} \lesssim K^{O(1)}\rho_0^{2\beta_0}
\]
\begin{equation}\label{eq:last-exit}
\kappa \ge 2\beta_0 r + O(\varepsilon).    
\end{equation}
So it remains to verify (\ref{eq:last-step}). Expanding the definitions, it reduces to the inequality
\[
2+ A^j\beta^* / A\lambda \le 2\beta_0 (1+ (2^j-1)/A\lambda)
\]
which would be implied by (recall that $\lambda =\log(2/\beta^*)$)
\begin{equation}\label{eq:reduced-to}
2 A\log(2/\beta^*)+ A^j\beta^* \le 2^j \beta_0.    
\end{equation}
We now select parameters as follows. Suppose that $\beta^*\le \beta_0/A^2$ and let $j(\beta^*) = [\log_A \beta_0/\beta^*]-1$. Then clearly $j(\beta^*)\ge 1$ and 
\[
A^{j(\beta^*)}\beta^* \le (\beta_0/\beta^*)A^{-1}\beta^*\le 2^{j(\beta^*)-1}\beta_0
\] 
and 
\[
2^{j(\beta^*)} \ge 2^{\log_A \beta_0/\beta^*-2} \ge \frac{1}{4}(\beta_0/\beta^*)^{1/\log_2 A}.
\]
This means that for $\beta^*\le \beta_0/A^2$ and $j = j(\beta^*)$, the relation (\ref{eq:reduced-to}) would follow from 
\begin{equation}\label{eq:reduce-to2}
(16 A\beta_0^{-1}) \log (2/\beta^*) \le (\beta_0/\beta^*)^{1/\log_2 A}    
\end{equation}
Since a polynomial eventually dominates logarithm, this is satisfied for sufficiently small $\beta^*$. Explicitly, we have for any $x>1$, $t>0$ that $\log x = t^{-1} \log x^t \le t^{-1} x^t$. So we can use the bound
\[
\log (2/\beta^*) \le 2(\log_2A) (2/\beta^*)^{\frac{1}{2\log_2 A}}
\]
and some easy rearrangements to reduce (\ref{eq:reduce-to2}) to
\[
\beta^* \le \left( 32 A\beta_0^{-1}\log_2A  \right)^{-2\log_2 A} \beta_0^2/2.
\]
So by setting $\beta^*$ to be the right hand side we can then define $r = 2^{-j(\beta^*)}$ and conclude  that this choice satisfies requirement of Proposition \ref{prop:conditions-for-iteration}. So starting from $j=1$ and (\ref{eq:start-iteration}) we can use  Proposition \ref{prop:conditions-for-iteration} to show that (\ref{eq:iteration-condition}) holds for $j=j(\beta^*)$. 

We conclude that with this choice of parameters, at least one of the inequalities (\ref{eq:first-exit}), (\ref{eq:second-exit}), (\ref{eq:step-j-exit}) or (\ref{eq:last-exit}) must be satisfied. All of the expressions one right hand side are positive functions of $\gamma$, $\beta^*$ and $r$, so it follows that $\kappa \ge \kappa(\gamma) >0$, as desired. This completes the proof of Theorem \ref{thm:main-restated}.
\end{proof}

\bibliographystyle{amsplain0.bst}
\bibliography{main}
\end{document}